\documentclass[10pt]{amsart}
\usepackage{geometry} 
\usepackage{latexsym,amssymb,amsmath,amsfonts,amsthm,exscale}
\usepackage{graphicx, graphics}

\usepackage[bookmarks=true,colorlinks=true,citecolor=green]{hyperref}

\usepackage[ruled,vlined,lined,boxed,commentsnumbered,norelsize]{algorithm2e}

\usepackage{listings}
\usepackage{color,xcolor}
\usepackage{floatflt}
\usepackage{tikz}
\usepackage{url}
\usepackage{enumerate}  
\usepackage{enumitem}
\usepackage{fancyhdr}
\usepackage{mathrsfs}

 \usepackage{graphicx}
 \usetikzlibrary{shapes,snakes,arrows,patterns,trees}

\usepackage{amssymb}
\usepackage{amsthm}
\usepackage{amsmath}

\theoremstyle{plain}
\newtheorem{theorem}{Theorem}[section]

\newtheorem{remark}{Remark}[section]
\newtheorem{prop}[theorem]{proposition}

\newtheorem{example}{Example}

\definecolor{mygreen}{rgb}{0,0.6,0}
\definecolor{mygray}{rgb}{0.5,0.5,0.5}
\definecolor{mymauve}{rgb}{0.58,0,0.82}
\usepackage{lipsum}
\makeatletter
\g@addto@macro{\endabstract}{\@setabstract}
\newcommand{\authorfootnotes}{\renewcommand\thefootnote{\@fnsymbol\c@footnote}}%
\makeatother

\DeclareMathOperator{\bpi}{\boldsymbol{\pi}}

\title[Multi-steepest descent algorithm]{A new approach to improve ill-conditioned parabolic optimal control problem via time domain decomposition}

\begin{document}
\maketitle

\begin{center}
  \LARGE 
   \par \bigskip

  \normalsize
  \authorfootnotes
  Mohamed Kamel RIAHI\footnote{Mohamed Kamel RIAHI : \email{\href{mailto:riahi@njit.edu}{\tt riahi@njit.edu}} \url{http://web.njit.edu/~riahi} }\textsuperscript{1}
  \par \bigskip

  \textsuperscript{1}
  Department of mathematical science, New Jersey Institute of Technology, University Heights Newark, New Jersey, USA.
  \par	\bigskip

  \today
\end{center}


\begin{abstract}
In this paper we present a new steepest-descent type algorithm for convex optimization problems. Our algorithm pieces the unknown into sub-blocs of unknowns and considers a partial optimization over each sub-bloc. In quadratic optimization, our method involves Newton technique to compute the step-lengths for the sub-blocs resulting descent directions. Our optimization method is fully parallel and easily implementable, we first presents it in a general linear algebra setting, then we highlight its applicability to a parabolic optimal control problem, where we consider the blocs of unknowns with respect to the time dependency of the control variable. The parallel tasks, in the last problem, turn``on" the control during a specific time-window and turn it ``off" elsewhere. We show that our algorithm significantly improves the computational time compared with recognized methods. Convergence analysis of the new optimal control algorithm is provided for an arbitrary choice of partition. Numerical experiments are presented to illustrate the efficiency and the rapid convergence of the method.
\end{abstract}


\section{Introduction}
	Typically the improvement of iterative methods is based on an implicit transformation of the original linear system in order to get a new system which has a condition number ideally close to one see~\cite{evans1983preconditioning,grote1997parallel,saad2003iterative} and references therein. This technique is known as preconditioning. Modern preconditioning techniques such as algebraic multilevel e.g.~\cite{malas2007incomplete,rude1993mathematical} and domain decomposition methods e.g.~\cite{quarteroni1999domain,toselli2005domain,bjorstad2004domain,lions2001parareal} attempt to produce efficient tools to accelerate convergence. Other techniques have introduced a different definition of the descent directions, for example, CG-method, GMRES, FGMRES, BFGS, or its limited memory version l-BFGS see for instance~\cite{saad2003iterative}. Others approaches (e.g.~\cite{MR2921091},\cite{Grippo86anonmonotone} and \cite{MR2904364} without being exhaustive) propose different formulas for the line-search in order to enhance the optimization procedure. 
	
	The central investigation of this paper is the enhancement of the iterations of the steepest descent algorithm via an introduction of a new formulation for the line-search. Indeed, we show how to achieve an optimal vectorized step-length for a given set of descent directions. Steepest descent methods~\cite{Cauchy1847Methode} are usually used for solving, for example, optimization problems, control with partial differential equations (PDEs) constraints and inverse problems. Several approaches have been developed in the cases of constrained and unconstrained optimization. 
	
	It is well-known that the algorithm has a slow convergence rate with ill-conditioned problems because the number of iterations is proportional to the condition number of the problem. The method of J.Barzila and J.Borwein~\cite{Barzilai88twopoint} based on two-point step-length for the steepest-descent method for approximating the secant equation avoids this handicap. Our method is very different because first, it is based on a decomposition of the unknown and proposes a set of bloc descent directions, and second because it is general where it can be coupled together with any least-square-like optimization procedure. 	 
	 
	 The theoretical basis of our approach is presented and applied to the optimization of a positive definite quadratic form. Then we apply it on a complex engineering problem involving control of system governed by PDEs. We consider the optimal heat control which is known to be ill-posed in general (and well-posed under some assumptions) and presents some particular theoretical and numerical challenges. We handle the ill-posedness degree of the heat control problem by varying the regularization parameter and apply our methods in the handled problem to show the efficiency of our algorithm. The distributed- and boundary-control cases are both considered.

	This paper is organized as follows: In Section 2, we present our method in a linear algebra framework to highlight its generality.  Section 3 is devoted to the introduction of the optimal control problem with constrained PDE on which we will apply our method. We present the Euler-Lagrange-system associated to the optimization problem and give the explicit formulation of the gradient in both cases of distributed- and boundary-control. Then, we present and explain the parallel setting for our optimal control problem. In Section 4, we perform the convergence analysis of our parallel algorithm. In Section 5, we present the numerical experiments that demonstrate the efficiency and the robustness of our approach. We make concluding remarks in Section 6. For completeness, we include calculus results in the Appendix.	

	Let $\Omega$ be a bounded domain in $\mathbb{R}^{3}$, and $\Omega_c\subset\Omega$, the boundary of $\Omega$ is denoted by $\partial\Omega$. We denote by $\Gamma\subset\partial\Omega$ a part of this boundary. We denote $\langle . , .\rangle_{2}$ (respectively $\langle . , .\rangle_{c}$ and $\langle . , .\rangle_{\Gamma}$) the standard $L^2(\Omega)$   (respectively $L^2(\Omega_c)$ and $L^2(\Gamma)$) inner-product that induces the $L^2(\Omega)$-norm $\|.\|_{2}$ on the domain $\Omega$ (respectively $\|\cdot\|_{c}$ on $\Omega_c$ and $\|\cdot\|_{\Gamma}$ on $\Gamma$).
	
	 In the case of finite dimensional vector space in $\mathbb{R}^m$, the scalar product $a^Tb$ of $a$ and $b$ (where $a^T$ stands for the transpose of $a$) is denoted by $\langle . , .\rangle_{2}$ too.
The scalar product with respect to the matrix $A$, i.e. $\langle x,Ax\rangle_{2}$ is denote by $\langle x,x\rangle_{A}$ and its induced norm is denoted by $\|x\|_{A}$. 
The transpose of the operator $A$ is denoted by $A^T$. The Hilbert space $L^2(0,T;L^2(\Omega_c))$ (respectively $L^2(0,T;L^2(\Gamma))$) is endowed by the scalar product $\langle . , .\rangle_{c,I}$ ( respectively $\langle . , .\rangle_{\Gamma,I}$) that induces the norm $\|.\|_{c,I}$ (respectively $\|.\|_{\Gamma,I}$). 
\section{Enhanced steepest descent iterations}
	The steepest descent algorithm minimizes at each iteration the quadratic function $q(x)=\|x-x^\star\|_A^2$, where $A$ is assumed to be a symbiotic positive definite (SPD) matrix and $x^\star$ is the minimum of $q$. The vector $-\nabla q(x)$ is locally the descent direction that yields the fastest rate of decrease of the quadratic form $q$. Therefore all vectors of the form $x+\theta\nabla q(x)$, where $\theta$ is a suitable negative real value, minimize $q$. The choice of $\theta$ is found by looking for the $\min_{s<0}q(x+s\nabla q(x))$ with the use of a line-search technique. In the case where $q$ is a quadratic form $\theta$ is given by $-{\|\nabla q(x)\|_2^2}\slash{\|\nabla q(x)\|_A^2}$. We recall in Algorithm~\ref{algoSD} the steepest descent algorithm; $Convergence$ is a boolean variable based on estimation of the residual vector $r^k<\epsilon$, where $\epsilon$ is the stopping criterion.
\vspace{.2in}
\begin{algorithm}[!htbp]
\SetAlgoLined
\LinesNumbered
\KwIn{ $x_{}^0$\;}
$k=0$\;
\While{\text{Convergence}}
{
   $r^{k}=\nabla q^{k}:=\nabla q(x^{k})$\;
   Compute $Ar^{k}$\;
   Compute $\theta^{k}=-\|r^{k}\|_2^2\slash \|r^{k}\|_A^2$\;
   $x^{k+1}=x^{k}+\theta^{k} r^{k}$\;
   $k=k+1$\;
}
\caption{Steepest descent.\label{algoSD}}
\end{algorithm}\vspace{.2in}

	Our method proposes to modify the stepe $5.$ of Algorithm~\ref{algoSD}. It considers the step-length $\theta\in\mathbb{R}_{-}\backslash\{0\}$ as a vector in $\mathbb{R}^{\hat n}_{-}\backslash\{{\bf0}\}$ where $\hat n$ is an integer such that $1\leq\hat n\leq size(x)$, we shall denote this new vector as $\Theta_{\hat n}$. 
	
	In the following, it is assumed that for a giving vector $x\in\mathbb{R}^{m}$, the integer $\hat n$ divides $m$ with null rest. In this context, let us introduce the identity operators $I_{\mathbb{R}^{m}}$ witch is an $m$-by-$m$ matrix and its partition (partition of unity) given by the projection operators $\{\bpi_n\}_{n=1}^{\hat n}$ : projectors from $\mathbb{R}^{m}$ into a set of canonical basis $\{e_i\}_i$. These operators are defined  for $1\leq n\leq \hat n$ by 
	$$\begin{array}{ll}
	\bpi_n&:\mathbb{R}^{m}\rightarrow\mathbb{R}^{\frac{m}{\hat n}}\\
	&x\mapsto \bpi_n(x)=\displaystyle\sum_{i=(n-1)\times\frac{m}{\hat n}+1}^{n\times\frac{m}{\hat n}} \langle e_i,x\rangle_2 \, e_i.	\end{array}
	$$
	 For reading conveniences, we define $\tilde x_n$ a vector in $\mathbb{R}^{m}$ such that $\tilde x_n:=\bpi_n(x)$. The concatenation of $\tilde x_n$ for all $1\leq n\leq \hat n$ is denoted by 
	$$\hat x_{\hat n}=\bigoplus_{n=1}^{\hat n} \bpi_n (x)=\bigoplus_{n=1}^{\hat n} \tilde x_{n}   \in \mathbb{R}^{m}.$$
We remark that $\bpi_n$ satisfy $\bigoplus_{n=1}^{\hat n} \bpi_n = I_{\mathbb{R}^{m}}$.

	 Recall the gradient $\nabla_{x} =(\frac{\partial}{\partial x_{1}},\dots,\frac{\partial}{\partial  x_{m}})^T$, and define the bloc gradient 
	 $\nabla_{\hat x_{\hat n}}=\left(\nabla_{\tilde x_{1}}^T,\dots,\nabla_{\tilde x_{\hat n}}^T \right)^T$, where obviously $\nabla_{\tilde x_{n}}^T=(\frac{\partial}{\partial x_{(n-1)\times\frac{m}{\hat n}+1}},\dots,\frac{\partial}{\partial  x_{n\times\frac{m}{\hat n}}})^T$.
	 In the spirit of this decomposition we investigate, in the sequel, the local descent directions as the bloc partial derivatives with respect to the bloc-variables $(\tilde x_n)_{n=1}^{n=\hat n}$. We aim, therefore, at finding $\Theta_{\hat n}=(\theta_1,\dots,\theta_{\hat n})^T\in\mathbb{R}^{\hat n}$ that ensures the $\min_{(\theta_n)_n<0}q \left(\hat  x_{\hat n}^{k}+\bigoplus_{n=1}^{\hat n} \theta_{n}\nabla_{\tilde x_{n}} q(\hat  x_{\hat n}^{k})\right)$.

We state hereafter a motivating result, which its proof is straightforward because the spaces are embedded. Let us first, denote by 
\begin{equation}\label{eqPhiAlg}
\begin{array}{ll}
\Phi_{\hat n}(\Theta_{\hat n}):&\mathbb R^{\hat n}\to\mathbb R_{+}\\
&\Theta_{\hat n}\mapsto q \left(\hat  x_{\hat n}^{}+\bigoplus_{n=1}^{\hat n} \theta_{n}\nabla_{\tilde x_{n}} q(\hat  x_{\hat n}^{})\right)
\end{array}\end{equation}
which is quadratic because $q$ is. 
\begin{theorem}\label{rmkmotivate}
 According to the definition of $\Phi_{\hat n}(\Theta_{\hat n})$ (see Eq.\eqref{eqPhiAlg}) we immediately have 
$$
\min_{\mathbb{R}^{p}} \Phi_{p}(\Theta_{p}) \leq \min_{\mathbb{R}^{q}} \Phi_{q}(\Theta_{q}) \quad \forall q<p.
$$
\end{theorem}

The new algorithm we discuss in this paper proposes to define a sequence $(\hat x^k_{\hat n})_k$ of vectors that converges to $x^{\star}$ unique minimizer of the quadratic form $q$. The update formulae reads:
$$\tilde x^{k+1}_{n} =  \tilde  x_{n}^{k}+ \theta_{n}^k\nabla_{\tilde x_{n}} q(\hat  x_{\hat n}^{k}),$$
where we recall that $\hat n$ is an arbitrarily chosen integer. Then $\hat x^{k+1}_{\hat n} = \bigoplus_{n=1}^{\hat n}\tilde x^{k+1}_{n}$.

	We shall explain now how one can accurately computes the vector step-length $\Theta_{\hat n}^k$ at each iteration $k$. It is assumed that $q$ is a quadratic form. From Eq.\eqref{eqPhiAlg} using the chain rule, we obtain the Jacobian vector $\Phi_{\hat n}'(\Theta_{\hat n})\in\mathbb{R}^{\hat n}$ given by
\begin{equation}\label{dev1}	
	\left(\Phi_{\hat n}'(\Theta_{\hat n})\right)_{j} = \left(\nabla_{\tilde x_{j}} q (\hat x_{\hat n}^{k})\right)^T \nabla_{\tilde x_j} q  \left(\hat  x_{\hat n}^{k}+\bigoplus_{n=1}^{\hat n} \theta_{n}\nabla_{\tilde x_{n}} q(\hat  x_{\hat n}^{k})\right)\in\mathbb{R},
	\end{equation}
	 and the Hessian matrix $\Phi_{\hat n}''(\Theta_{\hat n})\in\mathbb{R}^{\hat n\times\hat n}$ is given by 
	\begin{equation*}
	\left(\Phi_{\hat n}''(\Theta_{\hat n})\right)_{i,j}=\left(\nabla_{\tilde x_{j}} q (\hat x_{\hat n}^{k})\right)^T \Big( \nabla_{\tilde x_{i}}\nabla_{\tilde x_j} q  \left(\hat  x_{\hat n}^{k}+\bigoplus_{n=1}^{\hat n} \theta_{n}\nabla_{\tilde x_{n}} q(\hat  x_{\hat n}^{k})\right) \Big) \nabla_{\tilde x_{j}} q (\hat x_{\hat n}^{k}).
	\end{equation*}
	It is worth noticing that the matrix $\nabla_{\tilde x_{i}}\nabla_{\tilde x_j} q  \left(\hat  x_{\hat n}^{k}+\bigoplus_{n=1}^{\hat n} \theta_{n}\nabla_{\tilde x_{n}} q(\hat  x_{\hat n}^{k})\right)$ is a bloc portion of the Hessian matrix $A$. However if the gradient $\nabla_{\tilde x_{n}}q\in\mathbb{R}^{\frac{m}{\hat n}}$ assumes an extension by zero (denoted by $\widetilde\nabla_{\tilde x_{i}}q$) to $\mathbb{R}^{m}$ so the matrix $\Phi_{\hat n}''(\Theta_{\hat n})$ has therefore the simplest implementable form 
	\begin{equation}\label{dev2}
	\left(\Phi_{\hat n}''(\Theta_{\hat n})\right)_{i,j}=(\widetilde\nabla_{\tilde x_{j}} q (\hat x_{\hat n}^{k}))^T  A \widetilde\nabla_{\tilde x_{j}} q (\hat x_{\hat n}^{k}).
	\end{equation}
 We thus have the expansion  $\Phi_{\hat n}(\Theta_{\hat n}^{k})=\Phi_{\hat n}({\bf 0})+(\Theta_{\hat n}^{k})^T\Phi_{\hat n}({\bf 0})+\frac 1 2 (\Theta_{\hat n}^{k})^T\Phi_{\hat n}''({\bf 0})\Theta_{\hat n}^{k}$, with ${\bf 0}:=(0,..,0)^T\in\mathbb{R}^{\hat n}$. Then the vector $\Theta_{\hat n}^{k}$ that annuls the gradient writes: 
\begin{equation}\label{theta_algebra}
\Theta_{\hat n}^k= -\Phi_{\hat n}''({\bf 0})^{-1}\Phi_{\hat n}'({\bf 0}).
\end{equation}
Algorithm~\ref{algoSD} has therefore a bloc structure which can be solved in parallel. This is due to the fact that partial derivatives can be computed independently. The new algorithm is thus as follows (see Algorithm~\ref{LASD})

\vspace{.2in}
\begin{algorithm}[htbp]
\SetAlgoLined
\LinesNumbered
$k=0$\;
\KwIn{ $\hat x_{\hat n}^0\in\mathbb{R}^m$\;}
\While{\text{Convergence}}
{
  \ForAll{$1\leq n \leq\hat n$}
  { 
  	$\tilde x_{n}^{k}=\bpi_{n}(\hat x_{\hat n}^k)$\;
  	$r_n=\nabla_{\tilde x_{n}^{k} } q(\hat x_{\hat n}^k)$\;
	$resize(r_n)$ (i.e. extension by zero means simply project on $\mathbb{R}^m$)\;
  }
   Assemble $\Phi_{\hat n}'({\bf0})$ with element $ \left(\Phi_{\hat n}'({\bf0})\right)_j = r_j^Tr_j$ according to Eq.\eqref{dev1}\;
   Assemble $\Phi_{\hat n}''({\bf 0})$ with element $\left(\Phi_{\hat n}''({\bf 0})\right)_{i,j}= r_{i} ^T A  r_{j}$  according to Eq.\eqref{dev2}\;
   Compute $\Theta_{\hat n}^{k}$ solution of Eq.\eqref{theta_algebra}\;
   Update	$\hat x^{k+1}_{\hat n} =  \hat  x_{\hat n}^{k}+ \bigoplus_{n}\theta_{n}\nabla_{\tilde x_{n}} q(\hat  x_{\hat n}^{k})$\;
   $k=k+1$\;
}
\caption{Enhanced steepest descent.\label{LASD}}
\end{algorithm}\vspace{.2in}

%
\section{Application to a parabolic optimal control problem}\label{AtPOCP}
In this part we are interested in the application of Algorithm \ref{LASD} in a finite element computational engineering problem involving optimization with constrained PDE. In particular, we deal with the optimal control problem of a system, which is governed by the heat equation. We shall present two types of control problems. The first concerns the distributed optimal control and the second concerns the Dirichlet boundary control. The main difference from the algorithm just presented in linear algebra is that the decomposition is applied on the time domain when the control. This technique is not classical, we may refer to a similar approaches that has been proposed for the time domain decomposition in application to the control problem, for instance~\cite{MR1931522,doi:10.1137/050647086,parareal_control_kamel} which basically they use a variant of the parareal in time algorithm~\cite{lions2001parareal}. 
\subsection{Distributed optimal control problem}\label{DistBCP}
Let us briefly present the steepest descent method applied to the following optimal control problem: find $v^\star$ such that
\begin{equation}\label{pb1}
  J(v^\star)=\min_{v\in L^2(0,T;L^2(\Omega_{c}))}   J(v),
\end{equation}
  where $J$ is a quadratic cost functional defined by
\begin{equation}\label{funcJ}
   J(v)= \frac{1}{2} \|  y(T)-y^{target}\|^{2}_{2} + \frac{\alpha}{2} \int_{I} \|v\|^{2}_{c} dt,
\end{equation}
where $y^{target}$ is a given target state and $y(T)$ is the state variable at time $T>0$ of the heat equation controlled by the  variable $v$ over $I:=[0,T]$. The Tikhonov regularization parameter $\alpha$ is introduced to penalize the control's $L^2$-norm over the time interval $I$. The optimality system of our problem reads: 
\begin{eqnarray}
&\left\{\begin{array}{ll}\label{fwd}
\partial_{t} y - \sigma\Delta y = \mathcal B v,\quad\text{on }I\times\Omega,\\
y(t=0)=y_{0}.
\end{array}\right.\\
&\left\{\begin{array}{ll}\label{bwd}
\partial_{t} p + \sigma\Delta p=0,\quad\text{ on }I\times\Omega,\\
p(t=T)= y(T)-y^{target}.\end{array}\right.\\
&\quad\nabla   J(v)=\alpha v+\mathcal{B}^Tp=0,\text{ on }I\times\Omega.\label{gradOmega}
\end{eqnarray}
In the above equations, the operator $\mathcal B$ is a linear operator that distributes the control in $\Omega_c$, obviously $\mathcal B$ stands for the indicator of $\Omega_c\subset\Omega$, the state variable $p$ stands for the Lagrange multiplier (adjoint state) solution of the backward heat equation Eq.\eqref{bwd}, Eq.\eqref{fwd} is called the forward heat equation.

\subsection{Dirichlet boundary optimal control problem}\label{DirichBCP}
In this subsection we are concerned with the PDE constrained Dirichlet boundary optimal control problem, where we aim at minimizing the cost functional $J_\Gamma$ defined by
\begin{equation}\label{funcJGamma}
   J_\Gamma(v_\Gamma)= \frac{1}{2} \|  y_{\Gamma}(T)-y^{target}\|^{2}_{2} + \frac{\alpha_{}}{2} \int_{I} \|v_\Gamma\|^{2}_{\Gamma} dt,
\end{equation}
where the control variable $v_{\Gamma}$ is only acting on the boundary $\Gamma\subset\partial\Omega$. Here too, $y^{target}$ is a given target state (not necessary equal the one defined in the last subsection ! ) and $y_{\Gamma}(T)$ is the state variable at time $T>0$ of the heat equation controlled by the variable $v_{\Gamma}$ during the time interval $I:=[0,T]$. As before $\alpha_{}$ is a regularization term. The involved optimality system reads 
 \begin{eqnarray}
&&\left\{\begin{array}{ccll}
\partial_{t} y_\Gamma - \sigma\Delta y_\Gamma &=&f&\quad\text{on } I\times\Omega\\
y_\Gamma&=&v_\Gamma &\quad\text{on } I\times\Gamma\\
y_\Gamma&=&g &\quad\text{on } I\times\{\partial\Omega\backslash\Gamma\}\\
y_\Gamma(0)&=&y_{0}&
\end{array}\right.\\
&&\left\{\begin{array}{ccll}\label{bwdGamma}
\partial_{t} p_\Gamma + \sigma\Delta p_\Gamma &=& 0& \text{ on } I\times\Omega\\
p_\Gamma&=& 0 & \text{ on } I\times\partial\Omega\\
p_\Gamma(T) &=& &y_\Gamma(T)-y^{target}
\end{array}\right.\\
&&\nabla J_\Gamma(v_\Gamma) = \alpha v_\Gamma - (\nabla p_\Gamma)^T\vec{n}=0 \quad\text{ on }\ I\times\Gamma,\label{gradGamma}
\end{eqnarray}
where $f\in L^2(\Omega)$ is any source term, $g\in L^2(\Gamma)$ and $\vec n$ is the outward unit normal on $\Gamma$. the state variable $p_{\Gamma}$ stands for the Lagrange multiplier (adjoint state) solution of the backward heat equation Eq.\eqref{bwdGamma}. Both functions $f$ and $g$ will be given explicitly for each numerical test that we consider in the numerical experiment section.

\subsection{Steepest descent algorithm for optimal control of constrained PDE }
In the optimal control problem, the evaluation of the gradient as it is clear in Eq.\eqref{gradOmega} (respectively \eqref{gradGamma}) requires the evaluation of the time dependent Lagrange multiplier $p$ (respectively $p_\Gamma$). This fact, makes the steepest descent optimization algorithm slightly differs from the  Algorithm \ref{algoSD} already presented.

	Let us denote by $k$ the current iteration superscript. We suppose that $v^0$ is  known. The first order steepest descent algorithm updates the control variable as follows:
\begin{equation}\label{updatev}
v^{k}=v^{k-1} + \theta^{k-1}\nabla J(v^{k-1}),\text{ for } k\geq 1, \quad\text{for the distributed control}
\end{equation}
respectively as 
\begin{equation}\label{updatev}
v^{k}_{\Gamma}=v^{k-1}_{\Gamma} + \theta_{\Gamma}^{k-1}\nabla J_{\Gamma}(v_{\Gamma}^{k-1}),\text{ for } k\geq 1,\quad\text{for the Dirichlet control}
\end{equation}
The step-length $\theta^{k-1}\in\mathbb{R}^-\backslash\{0\}$ in the direction of the gradient $\nabla J(v^{k-1})=\alpha v^{k-1}+\mathcal{B}^Tp^{k-1}$ (respectively $\nabla J_{\Gamma}(v_{\Gamma}^{k-1})=\alpha v_\Gamma -(\nabla p_\Gamma)^T\vec{n}$) is computed as :
 $$\theta^{k-1}=-\|\nabla  J(v^{k-1})\|^{2}_{c,I}
 \slash \|\nabla J(v^{k-1})\|^2_{\nabla^2J}\quad\text{for the distributed control}.$$ 
 respectively as
 $$\theta^{k-1}_{\Gamma}=-\|\nabla  J_{\Gamma}(v_{\Gamma}^{k-1})\|^{2}_{c,I}
 \slash \|\nabla J_{\Gamma}(v_{\Gamma}^{k-1})\|^2_{\nabla^2J_{\Gamma}}\quad\text{for the Dirichlet control}.$$  
 The above step-length $\theta^{k-1}$ (respectively $\theta_{\Gamma}^{k-1}$) is optimal (see e.g. \cite{MR680778}) in the sense that it minimizes the functional $\theta\rightarrow J(v^{k-1}+\theta\nabla J(v^{k-1}))$ (respectively $\theta\rightarrow J_{\Gamma}(v_{\Gamma}^{k-1}+\theta\nabla J_{\Gamma}(v_{\Gamma}^{k-1}))$). The rate of convergence of this technique is $\big(\frac{\underline\kappa-1}{\underline\kappa+1}\big)^{2}$, where $\underline\kappa$ is the condition number of the quadratic form, namely the Hessian of the cost functional $J$ (respectively $J_{\Gamma}$).

\subsection{Time-domain decomposition algorithm}
\label{riahi_mini_15_sec:3}
	Consider $\hat n$ subdivisions of the time interval $I=\displaystyle\cup_{n=1}^{\hat n} I_{n}$, consider also the following convex cost functional $ J$: 
\begin{eqnarray}\label{funcJplusieursvariable1}
  J(v_{1},v_{2},..,v_{\hat n})=\frac 1 2 \|\mathcal Y(T)-y^{target}\|_{2}^2 +\frac \alpha 2 \sum_{n=1}^{\hat n}\int_{I_{n}}\|v_{n}\|^2_{c} dt,\\
\label{funcJplusieursvariable2}    J_{\Gamma}(v_{1,\Gamma},v_{2,\Gamma},..,v_{\hat n,\Gamma})=\frac 1 2 \|\mathcal Y_{\Gamma}(T)-y^{target}\|_{2}^2 +\frac{\alpha_{}}{2} \sum_{n=1}^{\hat n}\int_{I_{n}}\|v_{n}\|^2_{\Gamma} dt,
\end{eqnarray}
where $v_{n}, n=1,...,\hat n$ are control variables with time support  included in $I_{n}, n=1,...,\hat n$. The state $\mathcal Y(T)$ (respectively $\mathcal Y_{Gamma}$) stands for the sum of state variables $\mathcal Y_{n}$ (respectively $\mathcal Y_{n,\Gamma}$) which are time-dependent state variable solution to the heat equation controlled by the variable $v_n$ (respectively $v_{n,\Gamma}$). . Obviously because the control is linear the state $\mathcal{Y}$ depends on the concatenation of controls $v_{1},v_{2},..,v_{\hat n}$ namely $v=\sum_{n=1}^{n=\hat n} v_n$.

	Let us define $\Theta_{\hat n}:=(\theta_{1},\theta_{2},...,\theta_{\hat n})^T$ where $\theta_{n}\in\mathbb R_{-}\backslash\{0\}$. For any admissible control $w=\sum_{n}^{\hat n} w_{n}$, we also define $\varphi_{\hat n}(\Theta_{\hat n}):=   J(v+\sum_{n=1}^{\hat n}\theta_nw_n)$, which is quadratic. We have:
\begin{equation}\label{variation}
\varphi_{\hat n}(\Theta_{\hat n})=\varphi_{\hat n}({\bf 0})+\Theta_{\hat n}^T\nabla\varphi_{\hat n}({\bf 0})+\frac 1 2\Theta_{\hat n}^T\nabla^2\varphi_{\hat n}({\bf 0})\Theta_{\hat n},
\end{equation}
where ${\bf 0}=(0,...,0)^T$. Therefore we can write $\nabla\varphi_{\hat n}(\Theta_{\hat n})\in\mathbb R^{\hat n}$ as $\nabla\varphi_{\hat n}(\Theta_{\hat n})=D(v,w)+H(v,w)\Theta_{\hat n}$, where the Jacobian vector and the Hessian matrix are given respectively by:
\begin{eqnarray*}
D(v,w)&:=&(\langle\nabla  J(v),\bpi_{1}(w)\rangle_{c},\dots,\langle\nabla  J(v),\bpi_{\hat n}(w)\rangle_{c})^T\in\mathbb R^{\hat n},\\
H(v,w)&:=&(H_{n,m})_{n,m}, \text{ for }\   H_{n,m}=\langle \bpi_{n}(w),\bpi_{m}(w)\rangle_{\nabla^2 J}.
\end{eqnarray*}
 Here, $(\bpi_n)$ is the restriction over the time interval $I_n$, indeed $\bpi_n(w)$ has support on $I_n$ and assumes extension by zero in $I$. The solution $\Theta_{\hat n}^\star$ of $\nabla\varphi_{\hat n}(\Theta_{\hat n})={\bf 0}$ can be written in the form:
\begin{equation}\label{thetastar}\Theta_{\hat n}^\star=-H^{-1}(v,w)D(v,w).\end{equation}
	 In the parallel distributed control problem, we are concerned with the following optimality system:
	\begin{eqnarray}\label{optim-System-Parallel}
	&&\left\{\begin{array}{ll}\label{fwdpDIS}
\partial_{t} \mathcal Y_{n} - \sigma\Delta \mathcal Y_{n} = \mathcal B v_{n},\quad\text{on }I\times\Omega,\\
\mathcal Y_{n}(t=0)= \delta_{n}^{0}y_{0}.
\end{array}\right.\\
&&\mathcal Y(T) = \sum_{n=1}^{\hat n}\mathcal Y_{n}(T) \\
&&\left\{\begin{array}{ll}\label{bwdDIS}
\partial_{t} \mathcal P + \sigma\Delta \mathcal P=0,\quad\text{ on }I\times\Omega,\\
\mathcal P(t=T)=\mathcal Y(T)-y^{target}.\end{array}\right.\\
&&\nabla J( \sum_{n=1}^{\hat n}v_{n})= \mathcal{B}^T\mathcal P+\alpha\sum_{n=1}^{\hat n} v_{n}=0,\text{ on }I\times\Omega.
	\end{eqnarray}
	where $\delta_{n}^{0}$ stands for the function taking value "$1$" only if $n=0$, else it takes the value "$0$". 
The Dirichlet control problem we are concerned with:
 \begin{eqnarray}
&&\left\{\begin{array}{ccll}\label{fwdpDIR}
\partial_{t} \mathcal Y_{n,\Gamma} - \sigma\Delta \mathcal Y_{n,\Gamma} &=&f&\quad\text{on } I\times\Omega\\
\mathcal Y_{n,\Gamma}&=&v_{n,\Gamma} &\quad\text{on } I\times\Gamma\\
\mathcal Y_{n,\Gamma}&=&g &\quad\text{on } I\times\{\partial\Omega\backslash\Gamma\}\\
\mathcal Y_{n,\Gamma}(0)&=&\delta_{n}^{0}y_{0}.&
\end{array}\right.\\
&&\mathcal Y_\Gamma(T) = \sum_{n=1}^{\hat n}\mathcal Y_{n,\Gamma}(T) \\
&&\left\{\begin{array}{ccll}\label{bwdDIR}
\partial_{t} \mathcal P_\Gamma + \sigma\Delta \mathcal P_\Gamma  &=& 0& \text{ on } I\times\Omega\\
\mathcal P_\Gamma &=& 0 & \text{ on } I\times\partial\Omega\\
\mathcal P_\Gamma (T) &=& &\mathcal Y_\Gamma(T)-y^{target}.
\end{array}\right.\\
&&\nabla J_\Gamma( \sum_{n=1}^{\hat n}v_{n,\Gamma}) = -\big(\nabla \mathcal P_\Gamma\big)^T\vec{n} + \alpha  \sum_{n=1}^{\hat n} v_{n,\Gamma}  =0 \quad\text{ on }\ I\times\Gamma.
\end{eqnarray}

	The resolution of Eqs.~\eqref{fwdpDIS} and \eqref{fwdpDIR} with respect to $n$ are fully performed in parallel over the time interval $I$. It is recalled that the superscript $k$ denotes the iteration index. The update formulae for the control variable $v^k$ is given by: 
	$$
	v^k_n = v^{k-1}_n + \theta^{k-1}_n 
	\mathcal{B}^T\mathcal P^{k-1}+\alpha\sum_{n=1}^{\hat n} v_{n}^{k-1} .$$
	respectively as 
	$$
	v^k_{n,\Gamma} = v^{k-1}_{n,\Gamma} + \theta^{k-1}_{n,\Gamma}  -\big(\nabla \mathcal P^{k-1}_\Gamma\big)^T\vec{n} + \alpha_{} \sum_{n=1}^{\hat n}v^{k-1}_{n,\Gamma} .
	$$
	
	\begin{quotation}
		We shall drop in the following the index $_{\Gamma}$ of the cost functional $J$. This index would be only used to specify which cost function is in consideration. unless the driven formulation apply for distributed as well as boundary control.
		\end{quotation}

	We show hereafter how to assemble vector step-length $\Theta_{\hat n}^k$ at each iteration. For the purposes of notation we denote by $H_k$ the $k$-th iteration of the Hessian matrix $H(\nabla  J(v^k),\nabla  J(v^k))$ and by $D_k$ the $k$-th iteration of the Jacobian vector $D(\nabla  J(v^k),\nabla  J(v^k))$. The line-search is performed with quasi-Newton techniques that uses at each iteration $k$ a Hessian matrix $H_{k}$ and Jacobian vector $D_{k}$ defined respectively by:   
\begin{eqnarray}\label{eqDk}
D_{k}&:=&\Big(\langle \nabla  J(v^k),\bpi_{1}\big(\nabla  J(v^k)\big)\rangle_{c},..,\langle\nabla  J(v^k),\bpi_{\hat n}\big(\nabla  J(v^k)\big)\rangle_{c}\Big)^T,\\
 (H_{k})_{n,m}&:=&\langle \bpi_{n}\big(\nabla  J(v^{k})\big),\bpi_{m}\big(\nabla  J(v^{k})\big)\rangle_{\nabla^2 J}.\label{eqHk}
 \end{eqnarray}
  The spectral condition number of the Hessian matrix $\nabla^2J$ is denoted as: $\underline\kappa=\underline\kappa(\nabla^2J):=\lambda_{max} \lambda_{min}^{-1}$,
 with $\lambda_{max}:=\lambda_{max}(\nabla^2J)$ the largest eigenvalue of $\nabla^2J$ and $\lambda_{min}:=\lambda_{min}(\nabla^2J)$ its smallest eigenvalue. 
 
 	According to ~Eq.\eqref{thetastar} we have 
\begin{equation}\label{tta}
\Theta_{\hat n}^k = -H_k^{-1}D_k.
\end{equation}
From Eq.\eqref{variation} we have:
\begin{equation}\label{Jvar}
 J(v^{k+1})= J(v^k)+(\Theta_{\hat n}^k)^TD_k + \frac 1 2 (\Theta_{\hat n}^k)^T\, H_{k}\, \Theta_{\hat n}^k.
\end{equation}
Our parallel algorithm to minimize the cost functional Eq.\eqref{funcJplusieursvariable1} and \eqref{funcJplusieursvariable2}, is stated as follows (see Algorithm~\ref{ESDDOC}).
\vspace{.2in}
\begin{algorithm}[htbp]
\SetAlgoLined\LinesNumbered\ShowLn
\KwIn{$v^0$}
\While{Convergence}
{
\ForAll{$1\geq n\geq \hat n$}
{
 \label{step1}  Solve $\mathcal Y_{n}(T)(v_{n}^k)$ of~Eq.\eqref{fwdpDIS}(respectively Eq.\eqref{fwdpDIR}) {\bf in parallel} for all $1\leq n \leq\hat n$\;
}

Compute $\mathcal P(t)$ with the backward problem according to Eq.\eqref{bwdDIS} (respectively Eq.\eqref{bwdDIR}) \label{step5}\;
\ForAll{$1\geq n\geq \hat n$}
{
\label{step2}Compute  $(D_{k})_{n}$ of Eq.\eqref{eqDk} {\bf in parallel} for all $1\leq n \leq\hat n$\;
}

Gather $(D_{k})_{n}$ from processor $n$,  $2\leq n\leq \hat n$ to {\bf master processor}\label{step3}\;
Assemble the Hessian matrix $H_{k}$ according to Eq.\eqref{eqHk} with {\bf master processor}\label{step4}\;

Compute the inversion of $H_{k}$ and calculate $\Theta_{\hat n}^{k}$ using Eq.\eqref{tta}\label{step6}\;
Broadcast $\theta_{n}^k$ from {\bf master processor} to all slaves processors\label{step7}\;
Update time-window-control variable $v_{n}^{k+1}$ {\bf in parallel} as  \label{step8}:
$$
v_n^{k+1} = v_n^k + \theta_n^k \bpi_n\big(\nabla  J(v^k)\big)\quad\text{ for all } 1\leq n \leq\hat n,
$$
 and go to~step~\ref{step1}\;
 $k=k+1$\;

}\caption{\bf Enhanced steepest descent algorithm for the optimal control problem.}\label{ESDDOC}
\end{algorithm}

	Since $(v_n)_n$ has disjoint time-support, thanks to the linearity, the notation $e_n\big(\nabla  J(v^k)\big)$ is nothing but $\nabla  J(v^k_n)$, where $v^k$ is the concatenation of $v_1^k,\dots,v^k_{\hat n}$. 
In Algorithm~\ref{ESDDOC} steps \ref{step3}, \ref{step4}, \ref{step6}, \ref{step7} and \ref{step8} are trivial tasks in regards to computational effort.

\section{Convergence analysis of Algorithm~\ref{ESDDOC}}
\label{riahi_mini_15_sec:4}
	This section provides the proof of convergence of Algorithm~\ref{ESDDOC}. In the sequel, we suppose that $\|\nabla  J(v^k)\|_{c}$ does not vanish; otherwise the algorithm has already converged. 

\begin{prop}
The increase in value of the cost functional $J$ between two successive controls $v^k$ and $v^{k+1}$ is bounded below by:
\begin{equation}\label{prop2}
 J(v^k)- J(v^{k+1})
\geq\frac {1}{2\kappa(H_k)}\frac{\|\nabla  J(v^k)\|^4_{c}}{\|\nabla  J(v^k)\|^2_{\nabla^2 J}}.
\end{equation}
\end{prop}
\begin{proof} 
	Using  Eq.\eqref{tta} and Eq.\eqref{Jvar}, we can write:
\begin{equation}\label{equivarphi}
 J(v^k)- J(v^{k+1})=\frac 1 2 D_{k}^TH_{k}^{-1}D_{k}.
\end{equation}
Preleminaries:
From the definition of the Jacobian vector $D_k$ we have
\begin{eqnarray*}
\|D_{k}\|^2_2&&=\sum_{n=1}^{\hat n}\langle\nabla  J(v^k),\bpi_{n}(\nabla  J(v^k))\rangle_{c}^2,\\
&&=\sum_{n=1}^{\hat n}\langle \bpi_{n}(\nabla  J(v^k)),\bpi_{n}(\nabla  J(v^k))\rangle_{c}^2,\\
&&=\sum_{n=1}^{\hat n}\|e_{n}(\nabla  J(v^k))\|^4_{c},\\
&&=\|\nabla  J(v^k)\|_{c}^4.\end{eqnarray*}
Furthermore since  $H_k$  is an SPD matrix we have
$\lambda_{\text{min}}(H_k^{-1})=\frac 1{\lambda_\text{max}(H_k)},$
from which we deduce:
$\frac 1 {\lambda_\text{min}(H_k^{})} \geq \frac 1{ \frac{1}{\hat n}\mathrm{1}^T_{\hat n}H_{k}\mathrm{1}_{\hat n}}.$
Moreover, we have:
\begin{eqnarray*}
	D_{k}^T H_{k}^{-1}D_{k}&&=\frac{D_{k}^T H_{k}^{-1}D_{k}}{\|D_{k}\|^2_2}\|D_{k}\|^2_2
	\geq \lambda_{min}(H_{k}^{-1})\|{D_{k}}\|^2_{2}\\
	&&=\lambda_{min}(H_k^{-1})\lambda_{min}(H_k)\frac{\|\nabla  J(v^k)\|_{c}^4}{\lambda_{min}(H_k)}\\
&&\geq\frac{\lambda_{min}(H_k)}{\lambda_{max}(H_k)}
	\frac{
	\|\nabla  J(v^k)\|_{c}^4
	}{
            \frac{1}{\hat n}\mathrm{1}^T_{\hat n}H_{k}\mathrm{1}_{\hat n}
	}
	\\
	&&=\frac{\hat n}{\kappa(H_k)} \|\nabla  J(v^k)\|^{-2}_{\nabla^2 J}\|\nabla  J(v^k)\|_{c}^4.
\end{eqnarray*}
	Since the partition number $\hat n$ is greater than or equal to $1$, we conclude that :
\begin{equation}\label{MajDHD}
D_{k}^T H_{k}^{-1}D_{k} \geq \frac{ \|\nabla  J(v^k)\|^{-2}_{\nabla^2 J}\|\nabla  J(v^k)\|_{c}^4}{\kappa(H_k)}.
\end{equation}
Hence, using Eq.\eqref{equivarphi} we get the stated result.
\end{proof}

\begin{theorem}\label{thmrate}
For any partition $\hat n$ of sub intervals, the control sequence $(v^k)_{k\geq1}$ of Algorithm~\ref{ESDDOC} converges to the optimal control $v^{k}$ unique minimizer of the quadratic functional $J$. Furthermore we have:
$$\|v^{k}-v^\star\|^2_{\nabla^2 J} \leq r^k\|v^{0}-v^\star\|^2_{\nabla^2 J},$$
where the rate of convergence $r:=\Big( 1-\frac{4\underline\kappa}{\kappa(H_k)(\underline\kappa+1)^2}\Big)$ satisfies $0\leq r<1.$
\end{theorem}

\begin{proof}
	 We denote by $v^\star$ the optimal control that minimizes $J$. The equality
$$
 J(v) =  J(v^\star) +\frac 1 2 \langle v-v^\star,v-v^\star\rangle_{\nabla^2 J}= J(v^\star) +\frac 1 2 \| v-v^\star\|_{\nabla^2 J}^2,
$$
holds for any control $v$; in particular we have:
\begin{eqnarray*}
 J(v^{k+1}) &=& J(v^\star) +\frac 1 2 \| v^{k+1}-v^\star\|_{\nabla^2 J}^2,\\
 J(v^{k}) &=& J(v^\star) +\frac 1 2 \| v^{k}-v^\star\|_{\nabla^2 J}^2.
\end{eqnarray*}
Consequently, by subtracting the equations above, we obtain
\begin{equation}\label{diffJ}
 J(v^{k+1})- J(v^{k})=\frac 1 2 \| v^{k+1}-v^\star\|_{\nabla^2 J}^2-\frac 1 2 \| v^{k}-v^\star\|_{\nabla^2 J}^2.\end{equation}
Since $J$ is quadratic, we have $\nabla^2 J(v^k-v^\star)=\nabla  J(v^k)$, that is $v^k-v^\star=(\nabla^2 J)^{-1}\nabla  J(v^k)$. Therefore we deduce:
\begin{eqnarray}\label{vv}
\|v^k-v^{\star}\|^2_{\nabla^2J} &&=\langle v^k-v^\star,v^k-v^\star\rangle_{\nabla^2 J} \\\nonumber
&&= \langle v^k-v^\star,\nabla^2J,v^k-v^\star\rangle_{c} \\\nonumber
&&= \langle (\nabla^2 J)^{-1}\nabla  J(v^k),\nabla^2 J,(\nabla^2 J)^{-1}\nabla  J(v^k)\rangle_{c} \\\nonumber
&&=\langle \nabla  J(v^k),(\nabla^2 J)^{-1},\nabla  J(v^k) \rangle_{c}\\\nonumber
&&=\|\nabla  J(v^k)\|^2_{(\nabla^2 J)^{-1}}.
\end{eqnarray}

Because of Eq.\eqref{equivarphi}, we also have
\begin{equation*} J(v^{k+1})- J(v^{k})=-\frac 1 2 D_{k}^TH_{k}^{-1}D_{k}.\end{equation*}

	Using Eq.\eqref{diffJ} and the above, we find that:
$$\|v^{k+1}-v^\star\|^2_{\nabla^2 J} = \|v^{k}-v^\star\|^2_{\nabla^2 J}-D_{k}^TH^{-T}_{k}D_{k}.$$
Moreover, according to Eqs~\eqref{MajDHD}-\eqref{vv}, we obtain the following upper bound:
\begin{eqnarray}
\|v^{k+1}-v^\star\|^2_{\nabla^2 J} 
&&\leq  \|v^{k}-v^\star\|^2_{\nabla^2 J} -  \frac{1}{\kappa(H_k)}\frac{\|\nabla  J(v^k)\|^4_{c}}{\|\nabla  J(v^k)\|^2_{\nabla^2 J}}\nonumber\\
&&\leq \|v^{k}-v^\star\|^2_{\nabla^2 J} \Big( 1 - \frac{1}{\kappa(H_k)}\frac{ \|\nabla  J(v^k)\|^4_{c} }{ \|\nabla  J(v^k)\|^2_{\nabla^2 J} \|\nabla  J(v^k)\|^2_{(\nabla^2 J)^{-1}} }\Big).\label{posmaj}
\end{eqnarray}
	Using the Kantorovich inequality~\cite{MR0053389,baksalary1991generalized} (see also The Appendix) :
\begin{equation}\label{kant-matrix}
\frac{ \|\nabla  J(v^k)\|^4_{c} }{ \|\nabla  J(v^k)\|^2_{\nabla^2 J} \|\nabla  J(v^k)\|^2_{(\nabla^2 J)^{-1}}}\geq \frac{4\lambda_{max}\lambda_{min}}{({\lambda_{max}}+{\lambda_{min}})^2}.
\end{equation}
Then
$$
	1 - \frac{1}{\kappa(H_k)}\frac{ \|\nabla  J(v^k)\|^4_{c} }{ \|\nabla  J(v^k)\|^2_{\nabla^2 J} \|\nabla  J(v^k)\|^2_{(\nabla^2 J)^{-1}} }
\leq 
 1-\frac{4\underline\kappa}{\kappa(H_k)(\underline\kappa+1)^2}.
$$
Finally we obtain the desired results for any partition to $\hat n$ subdivision, namely
$$
\|v^{k}-v^\star\|^2_{\nabla^2 J} \leq \Big( 1-\frac{4\underline\kappa}{\kappa(H_k)(\underline\kappa+1)^2}\Big)^{k}\|v^{0}-v^\star\|^2_{\nabla^2 J}.
$$
The proof is therefore complete.
\end{proof}
\begin{remark}
	Remark that the proof stands correct for the boundary control, need just to change the subscript "$c$" indicating the distributed control region $\Omega_c$, replace it by $"\Gamma"$ to indicate the boundary control on $\Gamma\subset\partial\Omega$ .
\end{remark}
\begin{remark}
	Remark that for $\hat n=1$, we immediately get the condition number $\kappa(H_k)=1$ and we recognize the serial steepest gradient method, which has convergence rate $\big(\frac{\underline\kappa-1}{\underline\kappa+1}\big)^2$.
\end{remark}	
	
	 It is difficult to pre-estimate the spectral condition number $\kappa(H_k)(\hat n)$ (is a function of $\hat n$) that play an important role and contribute to the evaluation of the rate of convergence as our theoretical rate of convergence stated. We present in what follows numerical results that demonstrate the efficiency of our algorithm,  Tests consider examples of well-posed and ill-posed control problem.  
\section{Numerical experiments}

	We shall present the numerical validation of our method in tow stages. In the first stage, we consider a linear algebra framework where we construct a random matrix-based quadratic cost function that we minimize using Algorithm~\ref{LASD}. In the second stage, we consider the two optimal control problems presented in sections~\ref{DistBCP} and in \ref{DirichBCP} for the distributed- and Dirchlet boundary- control respectively. In both cases we minimize a quadratic cost function properly defined for each handled control problem.   

\subsection{Linear algebra program}
This subsection treat basically the implementation of Algorithm~\ref{LASD}. The program was implemented using the scientific programming language Scilab~\cite{scilab}. We consider the minimization of a quadratic form $q$ where the matrix $A$ is an SPD $m$-by-$m$ matrix and a real vector $b\in\mathbb{R}^{m}\cap rank(A)$ are generated by hand (see below for their constructions). We aim at solving iteratively the linear system $Ax =b$, by minimizing 
	\begin{equation}\label{eqQ}
	q(x) = \frac{1}{2} x^T A x - x^Tb.
	\end{equation}
	
	Let us denote by $\hat n$ the partition number of the unknown $x\in\mathbb{R}^{m}$. The partition is supposed to be uniform and we assume that $\hat n$ divides $m$ with a null rest.
	
	 We give in Table~\ref{lstingScilab} a {\sc Scilab} function that builds the vector step-length $\Theta_{\hat n}^k$ as stated in Eq.~\eqref{theta_algebra}.
	 In the practice we randomly generate an SPD sparse matrix $A= (\alpha+\gamma m) I_{\mathbb{R}^{m}} +R$, where $0<\alpha<1$, $\gamma>1$, $I_{\mathbb{R}^{m}}$ is the $m$-by-$m$ identity matrix and $R$ is a symmetric $m$-by-$m$ random matrix. This way the matrix $A$ is symmetric and diagonally dominant, hence SPD. It is worthy noticing that the role of $\alpha$ is regularizing when rapidly vanishing eigenvalues of $A$ are generated randomly. This technique helps us to manipulate the coercivity of the handled problem hence its spectral condition number.

	 For such matrix $A$ we proceed to minimize the quadratic form defined in Eq.\eqref{eqQ} with several $\hat n$-subdivisions.

\begin{table}
\begin{lstlisting}[language=Scilab]
function [P]=Build_Hk(n,A,b,xk,dJk)
m=size(A,1);l=m/n;ii=modulo(m,n);
if ii~=0 then
    printf("Please chose an other n!");
    abort;
end
dJkn=zeros(m,n); Dk=[];
for i=1:n 
	dJkn((i-1)*l+1:i*l,i)=	dJk( (i-1)*l+1:i*l );
	Dk(i)=dJkn(:,i)'*(A*xk-b);
end
Hk=[,];
 for i=1:n
	for j=i:n
	  Hktmp=A*dJkn(:,j);
	  Hk(i,j)=dJkn(:,i)'*Hktmp;
	  Hk(j,i)=Hk(i,j);
	end
end
theta=-Hk\Dk;
P = eye(m,m);
	for i=1:n
		P( (i-1)*l+1:i*l , (i-1)*l+1:i*l )=theta(i).*eye(l,l);
	end
endfunction
\end{lstlisting}
\caption{Scilab function to build the vector step length, for the linear algebra program.}\label{lstingScilab}
\end{table}

	The improvement quality of the algorithm against the serial case $\hat n=1$ in term of iteration number is presented in Figure.~\ref{figalgcost}. In fact, the left hand side of Figure.~\ref{figalgcost} presents the cost function minimization versus the iteration number of the algorithm where several choices of partition on $\hat n$ are carried out. In the right hand side of the Figure.~\ref{figalgcost} we give the logarithmic representation of the relative error $\frac {\|x^k-x^\star\|_2}{\|x^\star\|_2}$, where $x^\star$ is the exact solution of the linear system at hand.

\begin{figure}[htbp]\centering{
 \includegraphics[scale=.4]{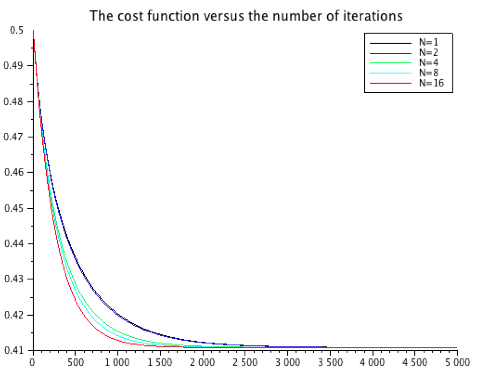}\hspace{.01in}
\includegraphics[scale=.4]{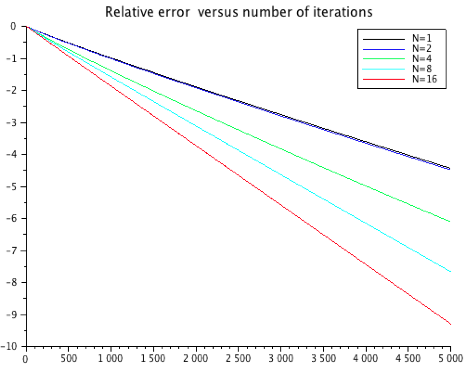}
}
\caption{Performance in term of iteration number: Several decomposition on $\hat n$. Results from the linear algebra Scilab program.}\label{figalgcost}
\end{figure}

\subsection{Heat optimal control program}
We discuss in this subsection the implementation results of Algorithm~\ref{ESDDOC} for the optimization problems presented in section~\ref{AtPOCP}. Our tests deal with the 2D-heat equation on the bounded domain $\Omega=[0,1]\times[0,1]$. We consider, three types of test problems in both cases of distributed and Dirichlet controls. Tests vary according to the theoretical difficulty of the control problem~\cite{MR1290658,MR2793831,survayControlWaveZuaZua}. Indeed, we vary the regularization parameter $\alpha$ and also change the initial and target solutions in order to handle more severe control problems as has been tested for instance in~\cite{MR1290658}. 

	Numerical tests concern the minimization of the quadratic cost functionals $J(v)$ and $J_{\Gamma}(v_{\Gamma})$ using Algorithm~\ref{ESDDOC}. It is well known that in the case $\alpha$ vanishes the control problem becomes an "approximated" controllability problem.  Therefore the control variable tries to produce a solution that reaches as close as it "can" the target solution. With this strategy, we accentuate the ill-conditioned degree of the handled problem. We also consider an improper-posed problems for the controllability approximation, where the target solution doesn't belong to the space of the reachable solutions. No solution exists thus for the optimization problem i.e. no control exists that enables reaching the given target !  
	
	For reading conveniences and in order to emphasize the role of the parameter $\alpha$ on numerical tests, we tag problems that we shall consider as $\mathcal P_i^\alpha$ where the index $i$ refers to the problem among $\{1,2,3,4\}$. The table below resumes all numerical test that we shall experiences


\vspace{.2in}
\begin{center}\footnotesize
\begin{tabular}{{|c||c|c|}}
\hline
--                                                           & Minimize $J(v)$ distributed control                                        & Minimize $J_{\Gamma}(v_\Gamma)$ boundary control  \\ \hline\hline
Moderate $\alpha=1\times 10^{-02}$   & well-posed problem                                                                &  ill-posed problem  \\
&  corresponding data in $(\mathcal P_1^{\alpha})$, $(\mathcal P_2^{\alpha}$)                                   & corresponding data in $(\mathcal P_3^{\alpha})$\\
 \hline\hline
Vanishing $\alpha=1\times 10^{-08}$   & ill-posed problem                                                                   & sever ill-posed problem \\
& corresponding data in $(\mathcal P_1^{\alpha})$ , $(\mathcal P_2^{\alpha}$)                                  & corresponding data in $(\mathcal P_3^{\alpha})$
\\ \hline\hline
Solution does not exist                         & sever ill-posed problem                                                         & sever ill-posed problem\\
& corresponding data in $(\mathcal P_4^{\alpha})$ &corresponding data in $(\mathcal P_4^{\alpha})$
 \\ \hline
\end{tabular}
\end{center}
\vspace{.2in}  
				 
	We suppose from now on that the computational domain $\Omega$ is a polygonal domain of the plane $\mathbb{R}^2$. We then introduce a triangulation $\mathcal{T}_{h}$ of $\Omega$; the subscript $h$ stands for the largest length of the edges of the tringles that constitute $\mathcal T_{h}$. The solution of the heat equation at a given time $t$ belongs to $H^1(\Omega)$. The source terms and other variables are elements of $L^2(\Omega)$. Those infinite dimensional spaces are therefore approximated with the finite-dimensional space $V_h$, characterized by $\mathbb P_1$ the space of the polynomials of degree $\leq 1$ in two variables $(x_1,x_2)$. We have $V_h:=\{u_h | \, u_h\in C^0(\overline{\Omega}), u_{h_{|K}}\in\mathbb P_1, \text{ for all } K\in\mathcal{T}_{h}\}$. In addition, Dirichlet boundary conditions (where the solution is in $H_{0}^1(\Omega)$ i.e. vanishing on boundary $\partial\Omega$) are taken into account via penalization of the vertices on the boundaries. The time dependence of the solution is approximated via the implicit Euler scheme. The inversion operations of matrices is performed by the {\sc umfpak} solver. We use the trapezoidal method in order to approximate integrals defined on the time interval.

	The numerical experiments were run using a parallel machine with 24 CPU's AMD with 800 MHz in a Linux environment. We code two FreeFem++~\cite{ffpp} scripts for the distributed and Dirichlet control. We use MPI library in order to achieve parallelism. 
	 
	 	 Tests that concern the distributed control problem are produced with control that acts on $\Omega_c\subset\Omega$, with $\Omega_c=[0,\frac 1 3]\times[0,\frac 1 3]$, whereas Dirichlet boundary control problem, the control acts on $\Gamma\subset\partial\Omega$, with $\Gamma=\{(x_1,x_2)\in\partial\Omega, | x_2=0\}$. The time horizon of the problem is fixed to $T=6.4$ and the small time step is $\tau=0.01$. In order to have a better control of the time evolution we put the diffusion coefficient $\sigma=0.01$.

\subsubsection{First test problem: Moderate Tikhonov regularization parameter $\alpha$}	
	We consider an optimal control problem on the heat equation. The control is considered first to be distributed and then Dirichlet. For the distributed optimal control problem we first use the functions 
\begin{equation}\tag{$\mathcal P_1^{\alpha}$}
\begin{split}
y_0(x_1,x_2)&= \exp\big(-\gamma 2\pi\big((x_1-.7)^2+(x_2-.7)^2\big)\big)\\
y^{target}(x_1,x_2)&= \exp\big(-\gamma 2\pi\big((x_1-.3)^2+(x_2-.3)^2\big)\big),
\end{split}
\end{equation}
as initial condition and target solution respectively. The real valued $\gamma$ is introduced to force the Gaussian to have support strictly included in the domain and verify the boundary conditions. The aim is to minimize the cost functional defined in Eq.~\eqref{funcJ}. 
\begin{figure}[htbp]
   \begin{minipage}[c]{.46\linewidth}
      \includegraphics[width=7cm,height=7cm]{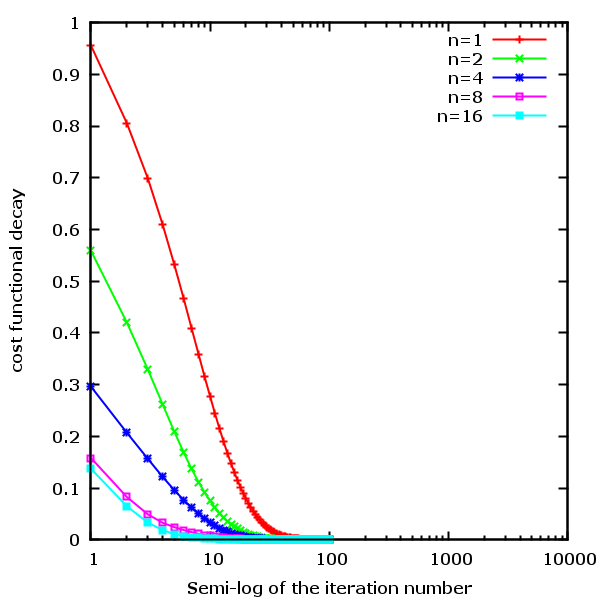}
   \end{minipage} \hfill
   \begin{minipage}[c]{.46\linewidth}
     \includegraphics[width=7cm,height=7cm]{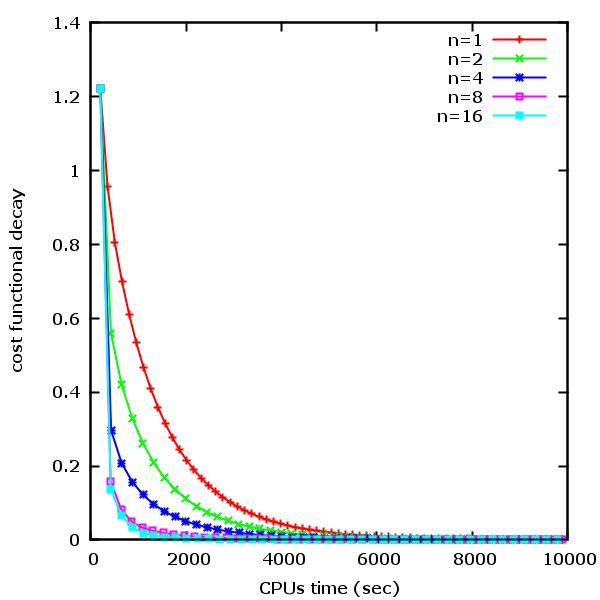}
   \end{minipage}
   \caption{ First test problem, for $\mathcal P_{1}^{\alpha}$: Normalized and shifted cost functional values versus iteration number (left) and versus computational time (right) for several values of $\hat n$ (i.e. the number of processors used).}\label{fig1}\label{Jtest1dist}
\end{figure}
The decay of the cost function with respect to the iterations of our algorithm is presented in Figure.~\ref{Jtest1dist} on the left side, and the same results are given with respect to the computational CPU's time (in sec) on the right side. We show that the algorithm accelerates with respect to the partition number $\hat n$ and also preserves the accuracy of the resolution. Indeed, all tests independently of $\hat n$ always converge to the unique solution. This is in agreement with Theorem~\eqref{thmrate}, which proves the convergence of the algorithm to the optimal control (unique if it exists~\cite{MR0271512}) for an arbitrary partition choice $\hat n$.

\begin{figure}[htbp]
\begin{tabular}{cc}
   \includegraphics[width=6cm,height=4.2cm]{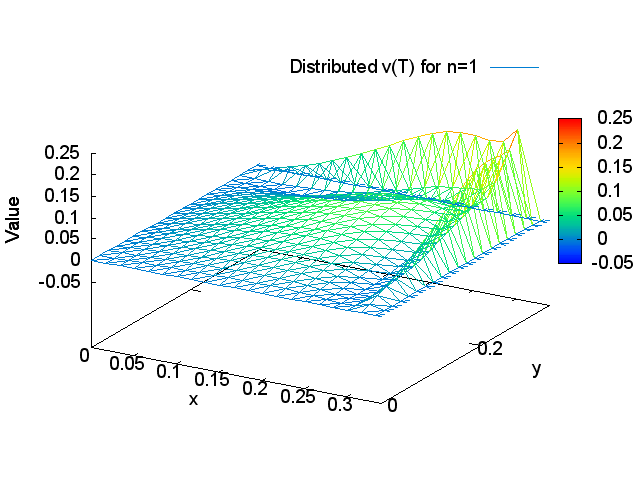} &   \includegraphics[width=6cm,height=4.2cm]{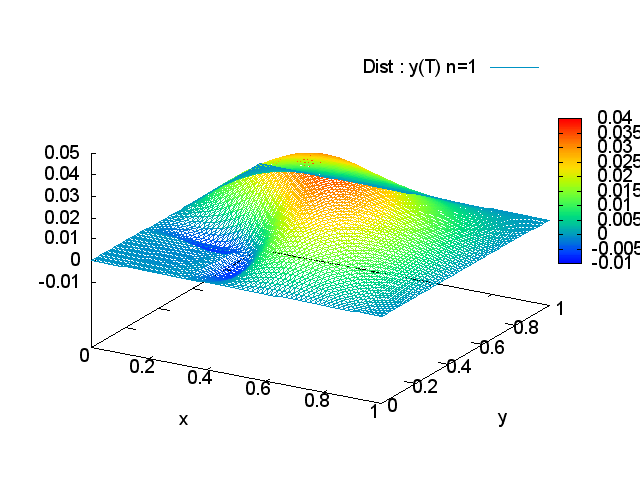} \\
   \includegraphics[width=6cm,height=4.2cm]{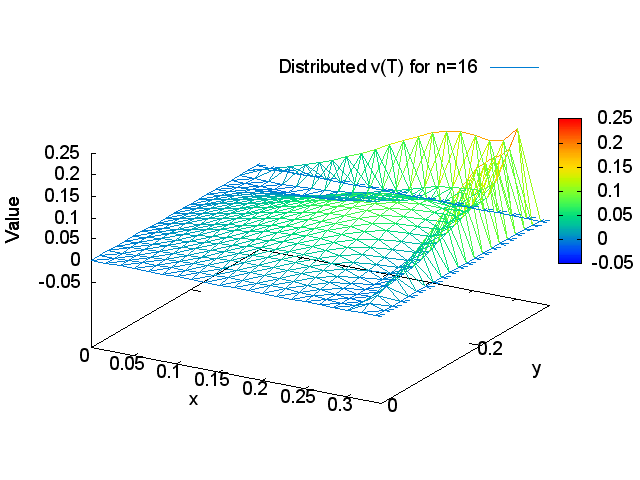} &   \includegraphics[width=6cm,height=4.2cm]{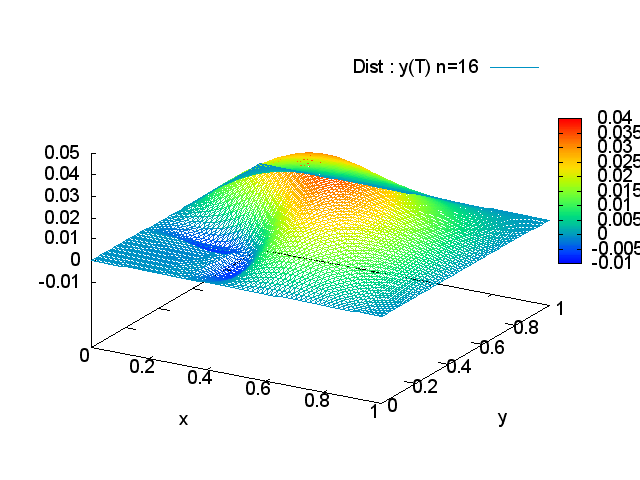} \\           
\end{tabular}
\caption{Snapshots in $\hat n=1,\,16$ of the distributed optimal control on the left columns and its corresponding controlled final state at time T: $y(T)$ on the right columns. The test case corresponds to the control problem $\mathcal P_1^\alpha$, where $\alpha$ is taken as $\alpha=1\times 10^{-02}$. Same result apply for different choice of $\hat n$.}
\label{snapshotTestDistGaussian}
\end{figure}

	We test a second problem with an a priori known solution of the heat equation. The considered problem has 
\begin{equation}\tag{$\mathcal P_2^{\alpha}$}
\begin{split}
y_0(x_1,x_2)&=\sin(\pi x_1)\sin(\pi x_2)\\
y^{target}(x_1,x_2)&=\exp(-2\pi^2\sigma  T)\sin(\pi x_1)\sin(\pi x_2),
\end{split}
\end{equation}
as initial condition and target solution respectively. Remark that the target solution is taken as a solution of the heat equation at time $T$. The results of this test are presented in Figure.~\ref{Jtest2dist}, which shows the decay in values of the cost functional versus the iterations of the algorithm on the left side and versus the computational CPU's time (in sec) on the right side. 

\begin{figure}[htbp]
   \begin{minipage}[c]{.46\linewidth}
      \includegraphics[width=7cm,height=7cm]{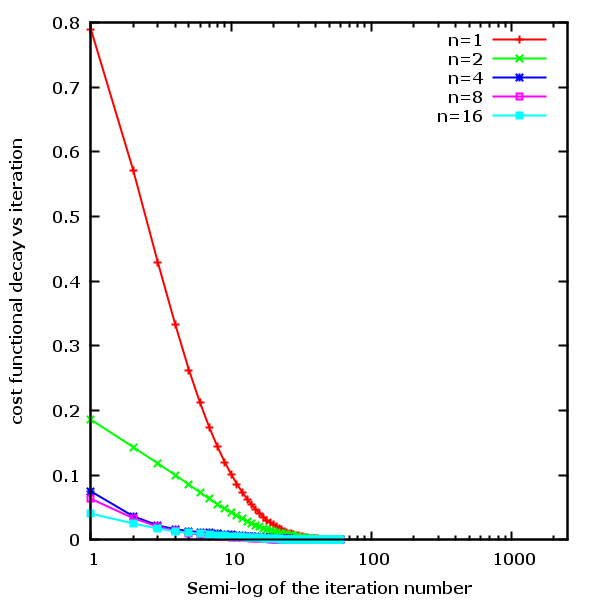}
   \end{minipage} 
   \begin{minipage}[c]{.46\linewidth}
      \includegraphics[width=7cm,height=7cm]{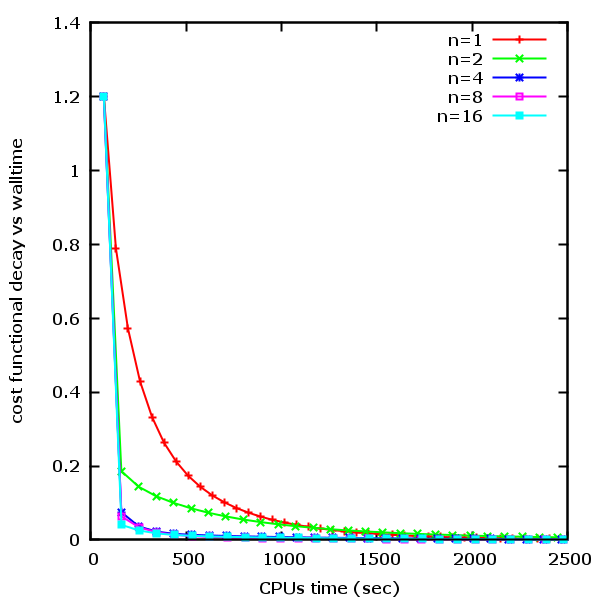}
   \end{minipage}
   \caption{ First test problem, for $\mathcal P_{2}^{\alpha}$: Normalized cost functional values versus computational CPU time  for several values of $\hat n$ (i.e. the number of processors used).} \label{Jtest2dist}
\end{figure}

\begin{figure}[htbp]
\begin{tabular}{cc}
   \includegraphics[width=6cm,height=4cm]{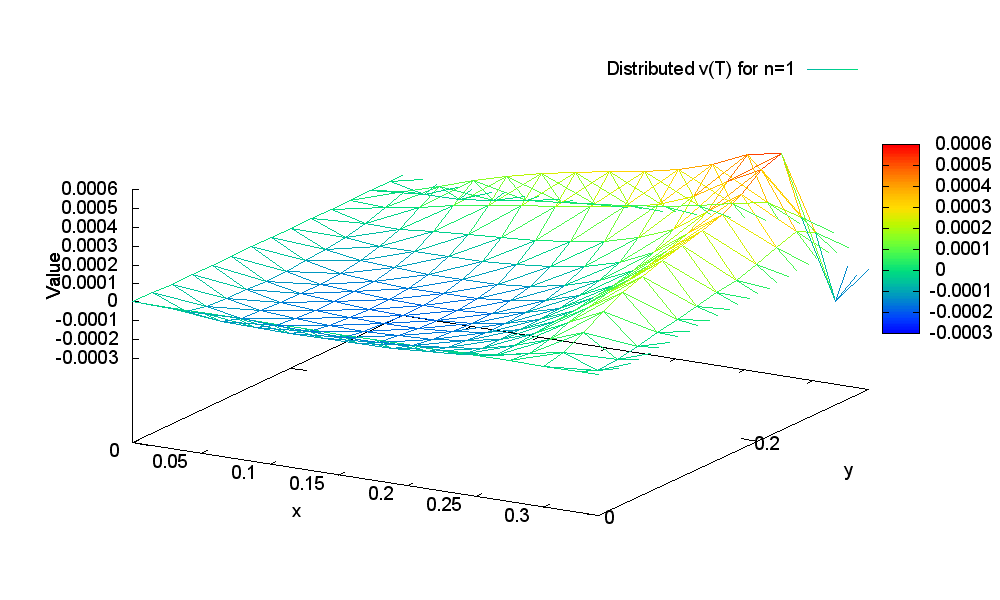} &   \includegraphics[width=6cm,height=4cm]{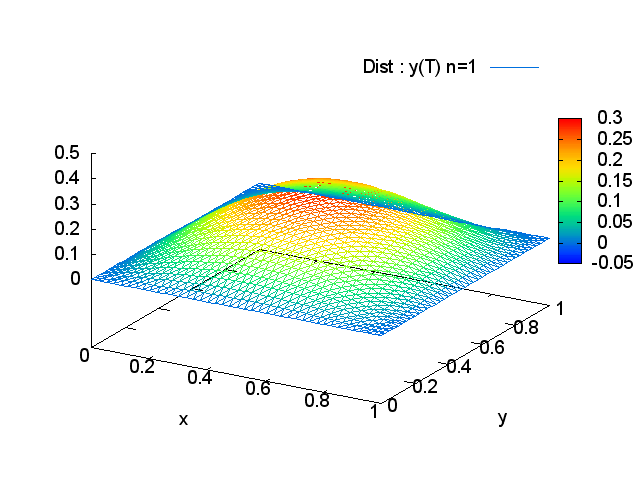} \\
   \includegraphics[width=6cm,height=4cm]{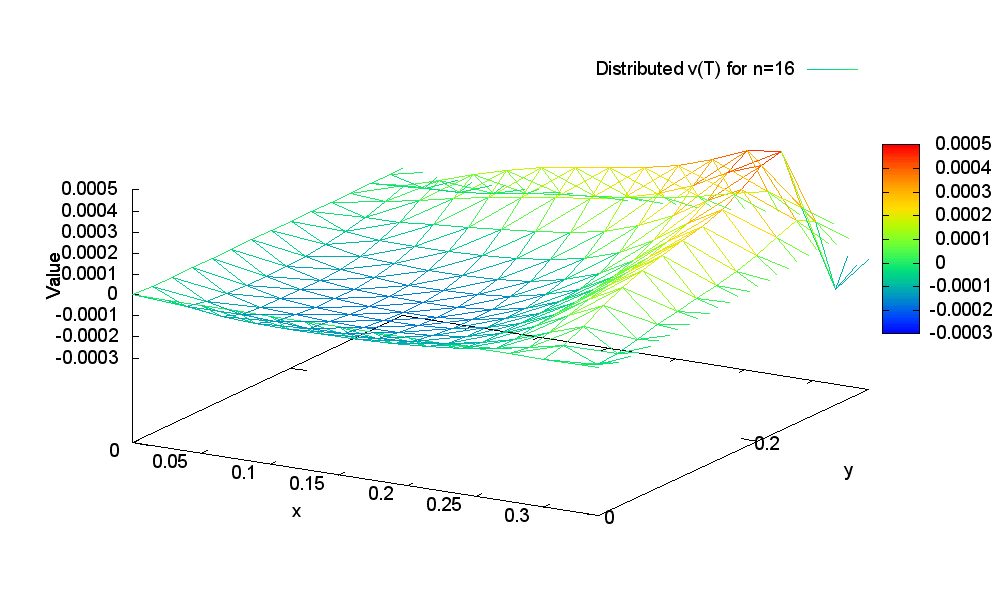} &   \includegraphics[width=6cm,height=4cm]{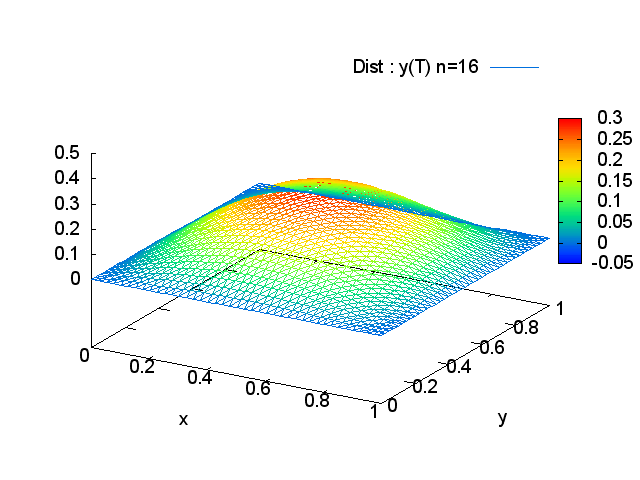} \\           
\end{tabular}
\caption{Snapshots in $\hat n=1,\,16$ of the distributed optimal control on the left columns and its corresponding controlled final state at time T: $y(T)$ on the right columns. The test case corresponds to the control problem $\mathcal P_2^\alpha$, where $\alpha=1\times 10^{-02}$. Same results apply for different choice of $\hat n$.}
\label{snapshotTestDist2}
\end{figure}
We give in Figure.~\ref{snapshotTestDistGaussian} and Figure.~\ref{snapshotTestDist2} several rows value snapshots (varying the $\hat n$) of the control and its corresponding controlled final solution $y (T)$.  Notice the stability and the accuracy of the method with any choice of $\hat n$. In particular the shape of the resulting optimal control is unique as well as the controlled solution $y(T)$ doesn't depend on $\hat n$.  
 
	For the Dirichlet boundary control problem we choose the following functions as source term, initial condition and target solution: 
	\begin{equation}\tag{$\mathcal P_3^{\alpha}$}
	\begin{split}
	f(x_1,x_2,t) &= 3\pi^3\sigma \text{exp}(2\pi^2\sigma t)(sin(\pi x_1) + sin(\pi x_2))\\
	 y_0(x_1,x_2) &= \pi(sin(\pi x_1) + sin(\pi x_2))\\
	y^{target}(x_1,x_2)  &= \pi\text{exp}(2\pi^2\sigma)(sin(\pi x_1) + sin(\pi x_2)),
	\end{split}
	\end{equation}
	respectively. Because of the ill-posed character of this problem, its optimization leads to results with hight contrast  in scale. We therefore preferred to summarize the optimizations results in Table~\ref{TabDirichletcontrol} instead of Figures.

\begin{remark}
	Because of the linearity and the superposition property of the heat equation, it can be shown that problems ($\mathcal P_2^\alpha$ and $\mathcal P_3^\alpha$) mentioned above are equivalent to a control problem which has null target solution.
\end{remark}


%

\begin{table}[htbp]
\tiny
\centering
\begin{tabular}{llccccccc||}
   \hline\hline
Test problem & \multicolumn{3}{c}{Results} \\
   \hline\hline
   $\mathcal P_1^\alpha$  &$\alpha=1\times 10^{-02}$ & & & & & \\
& Quantity & $\hat n=1$ & $\hat n=2$  & $\hat n=4$ & $\hat n=8$ & $\hat n=16$ \\   \hline
&Number of iterations $k$& 100  & 68 & 63 & 49 & 27 \\
&walltime in sec  & 15311.6 & 15352.3 & 14308.7& 10998.2& 6354.56 \\
&$\|\mathcal{Y}^k(T)-y^{target}\|_2\slash\|y^{target}\|_2$ & 0.472113 & 0.472117 & 0.472111&0.472104 & 0.472102  \\
&$ \int_{(0,T)}\|v^k\|_c^2 dt$ & 0.0151685 &0.0151509 &0.0151727 & 0.0152016 & 0.015214\\ 
   \hline 
   $\mathcal P_2^\alpha$  &$\alpha=1\times 10^{-02}$ & & & & & \\
& Quantity & $\hat n=1$ & $\hat n=2$  & $\hat n=4$ & $\hat n=8$ & $\hat n=16$ \\   \hline
&Number of iterations $k$& 60 & 50 & 45 & 40 & 35 \\
&walltime in sec & 3855.21 & 3726.28 & 4220.92 & 3778.13 & 3222.78 \\
&${\|\mathcal{Y}^k(T)-y^{target}\|_2\slash\|y^{target}\|_2}$ & $8.26\times 10^{-08}$ & $8.26\times 10^{-08}$&  $8.15\times 10^{-08}$& $8.15\times 10^{-08}$& $8.14 \times 10^{-08}$ \\
&$\int_{(0,T)}\|v^k\|_c^2 dt$ & $1.68\times 10^{-07}$ & $1.68\times 10^{-07}$ &$1.72\times 10^{-07}$  & $1.72\times 10^{-07}$ &$1.72\times 10^{-07}$\\\hline
   \hline 
   $\mathcal P_2^\alpha$ &$\alpha=1\times 10^{-08}$ & & & & & \\
& Quantity & $\hat n=1$ & $\hat n=2$  & $\hat n=4$ & $\hat n=8$ & $\hat n=16$ \\   \hline
&Number of iterations $k$& 60 & 50 & 40 & 30 & 20 \\
&walltime in sec & 3846.23 & 4654.34 & 3759.98 & 2835.31 & 1948.4 \\
&${\|\mathcal{Y}^k(T)-y^{target}\|_2\slash\|y^{target}\|_2}$ & $3.93\times 10^{-08}$ & $1.14\times 10^{-08}$&  $ 5.87\times 10^{-09}$& $2.04\times 10^{-09}$& $1.76 \times 10^{-09}$ \\
&$\int_{(0,T)}\|v^k\|_c^2 dt$ & $5.42\times 10^{-07}$ & $4.13\times 10^{-06}$ &$2.97\times 10^{-04}$  & $3.64\times 10^{-03}$ &$2.51\times 10^{-03}$\\\hline
   \end{tabular}
   \caption{Results' summary of Algorithm~\ref{ESDDOC} applied on the distributed control problems $\mathcal P_1^\alpha$ and $\mathcal P_2^\alpha$.}\label{TabDistributecontrol}
 \end{table}  		
\begin{figure}[htbp]
\begin{tabular}{cc}
   \includegraphics[width=6cm,height=4cm]{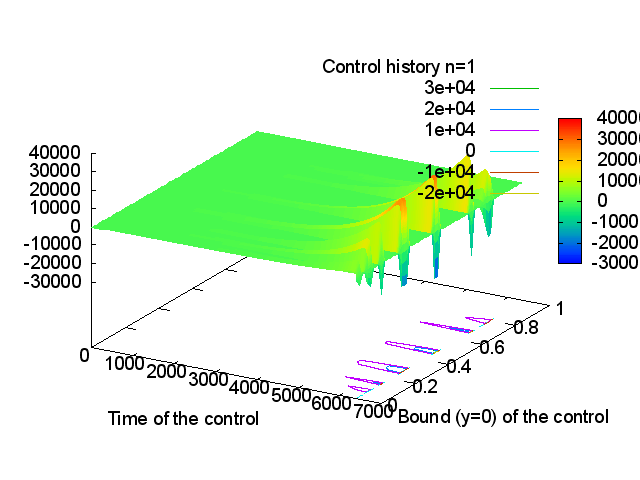} &   \includegraphics[width=6cm,height=4cm]{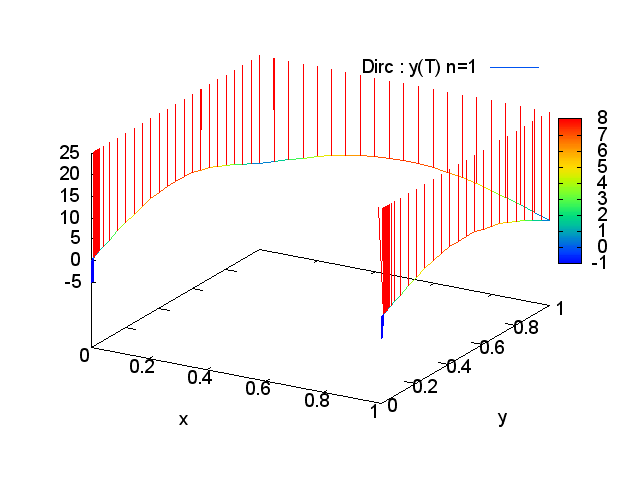} \\
   \includegraphics[width=6cm,height=4cm]{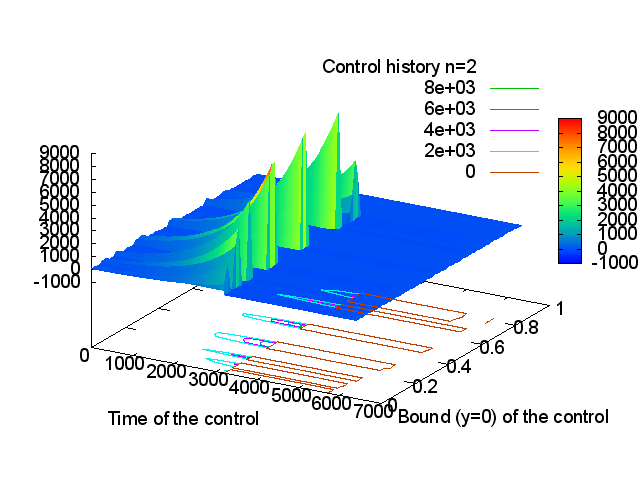} &   \includegraphics[width=6cm,height=4cm]{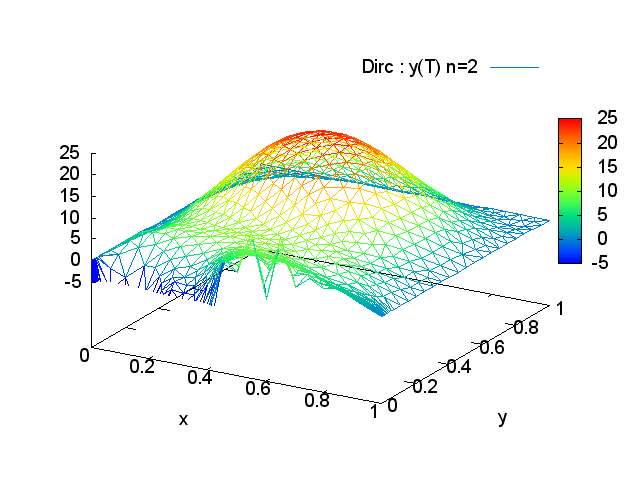} \\
   \includegraphics[width=6cm,height=4cm]{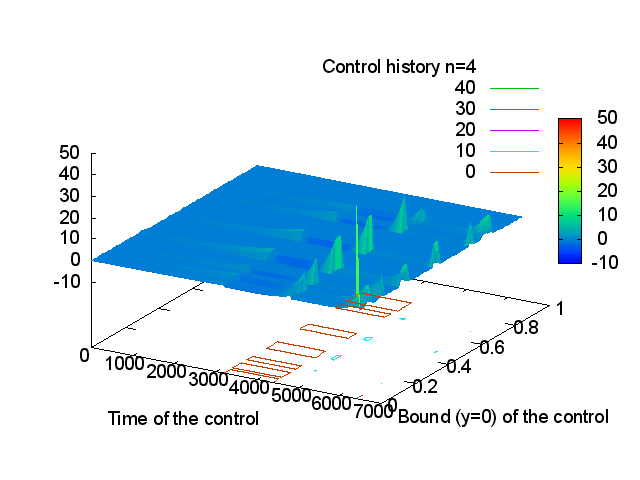} &   \includegraphics[width=6cm,height=4cm]{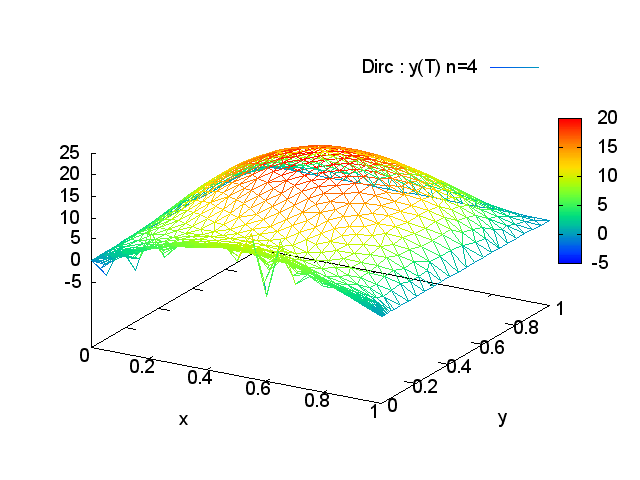} \\
   \includegraphics[width=6cm,height=4cm]{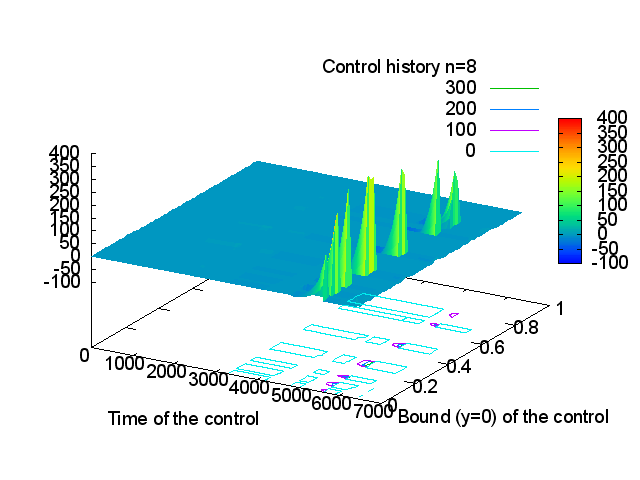} &   \includegraphics[width=6cm,height=4cm]{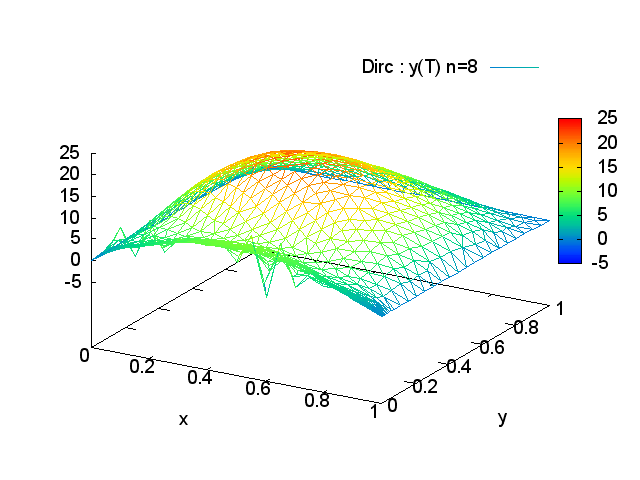} \\        
   \includegraphics[width=6cm,height=4cm]{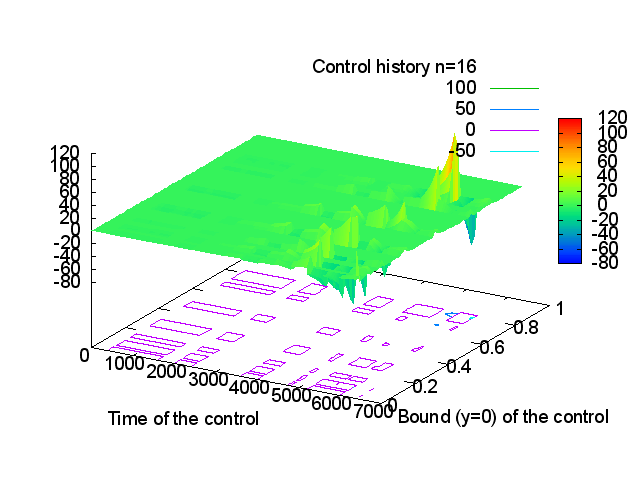} &   \includegraphics[width=6cm,height=4cm]{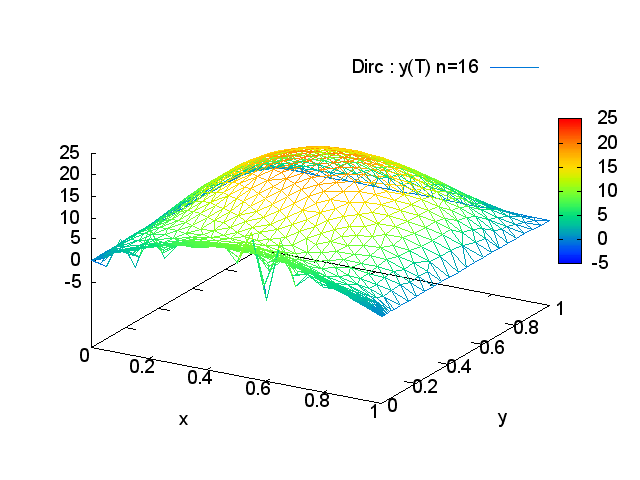} \\           
\end{tabular}
\caption{Several rows value snapshots in $\hat n$ of the Dirichlet optimal control on the left columns and its corresponding controlled final state at time T: $y(T)$ on the right columns. The test case corresponds to the control problem $\mathcal P_3^\alpha$, where $\alpha=1\times 10^{-02}$.}
\label{snapshot_Dirichlet1}
\end{figure}  

\subsubsection{Second test problem: vanishing Tikhonov regularization parameter $\alpha$}
In this section, we are concerned with the "approximate" controllability of the heat equation, where the regularization parameter $\alpha$ vanishes, practically we take $\alpha=1\times 10^{-08}$. In this case, problems $\mathcal P_2^\alpha$ and $\mathcal P_3^\alpha$, in the continuous setting are supposed to be well posed (see for instances~\cite{MR2086173,MR2841484}). However, may not be the case in the discretized settings; we refer for instance to~\cite{survayControlWaveZuaZua} (and reference therein) for more details. 
\begin{figure}[htbp]
   \begin{minipage}[c]{.46\linewidth}
 \includegraphics[width=7cm,height=7cm]{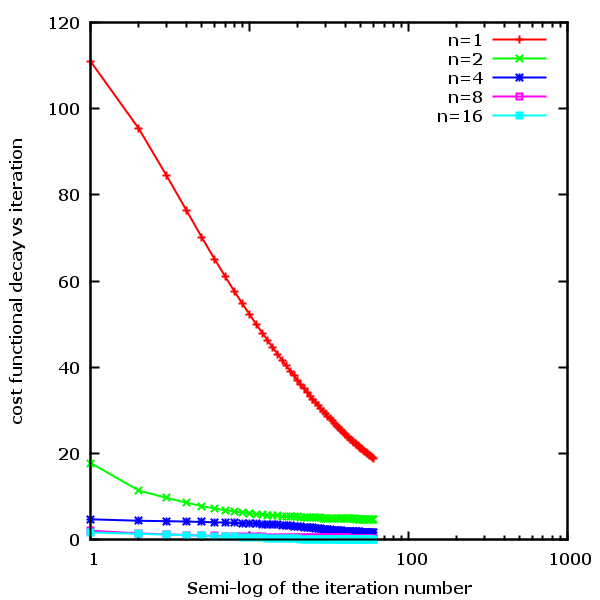}
   \end{minipage} 
   \begin{minipage}[c]{.46\linewidth}
     \includegraphics[width=7cm,height=7cm]{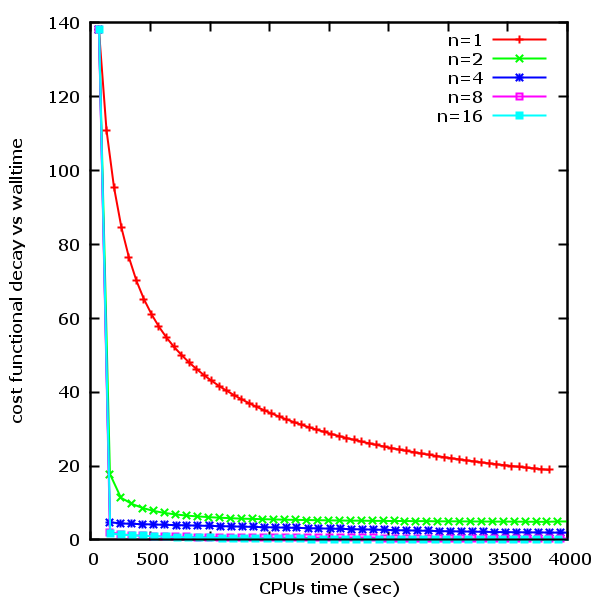}
   \end{minipage}
   \caption{ Normalized and shifted cost functional values versus computational CPU time  for several values of $\hat n$ (i.e. the number of processors used), Distributed control problem $\mathcal P_2^\alpha$ whith $\alpha=1\times 10^{-08}$.}\label{fig1}
\end{figure}

\begin{figure}[htbp]
\begin{tabular}{cc}
   \includegraphics[width=6cm,height=4cm]{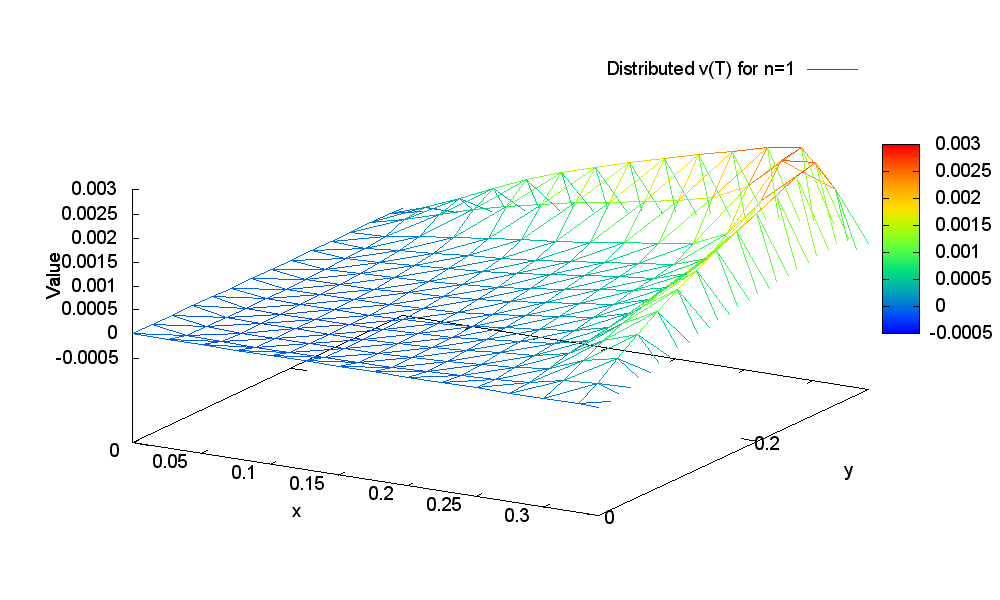} &   \includegraphics[width=6cm,height=4cm]{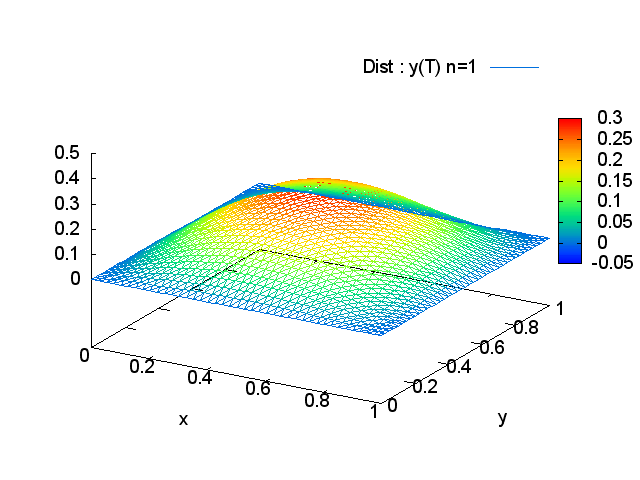} \\
   \includegraphics[width=6cm,height=4cm]{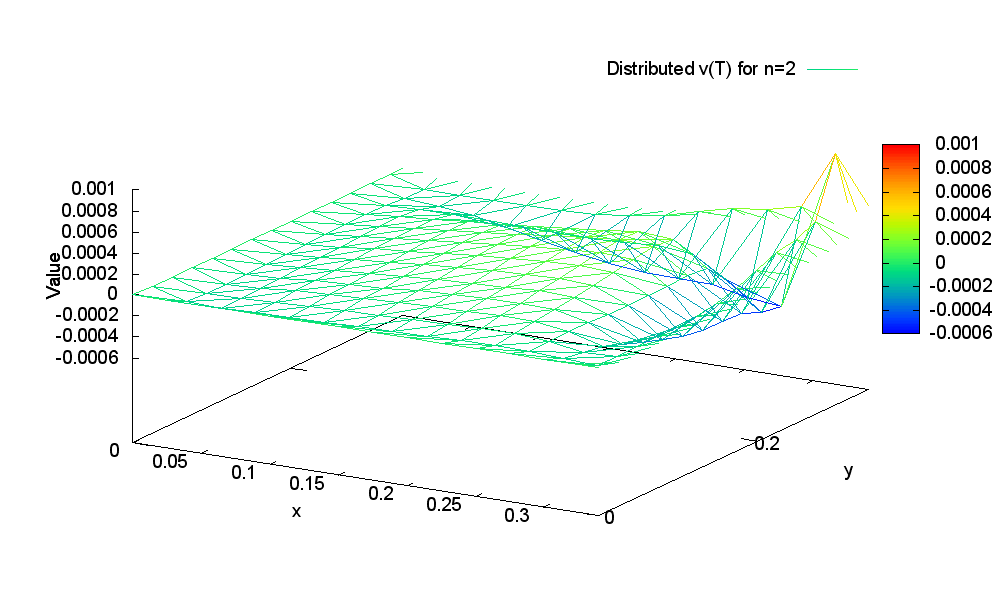} &   \includegraphics[width=6cm,height=4cm]{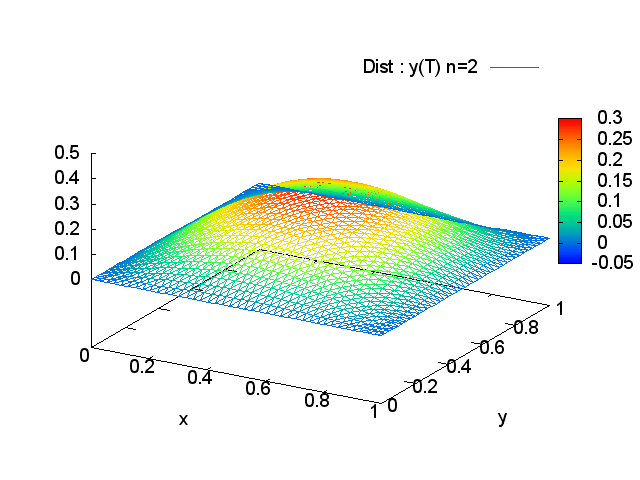} \\
   \includegraphics[width=6cm,height=4cm]{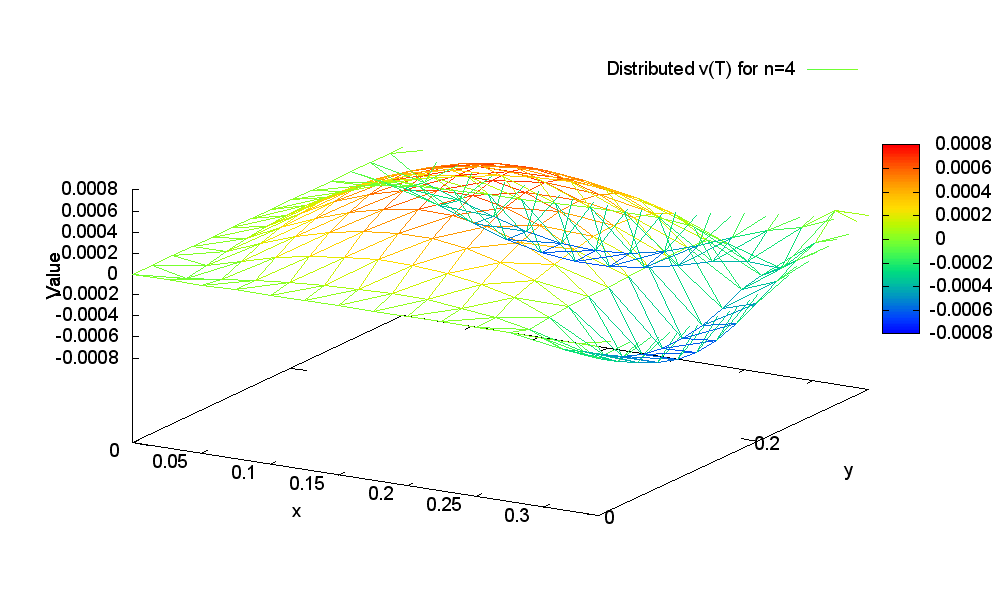} &   \includegraphics[width=6cm,height=4cm]{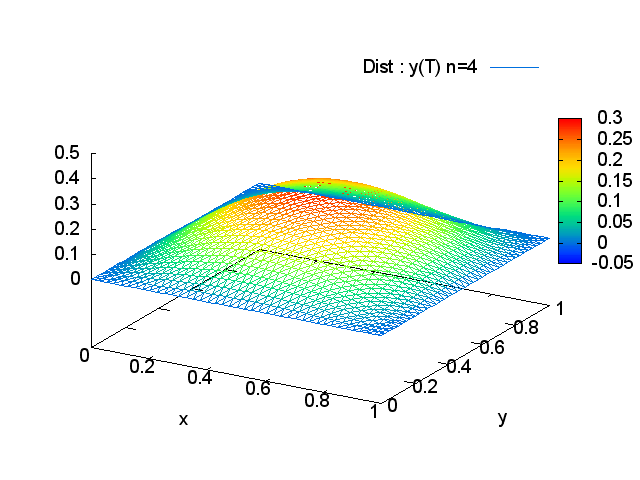} \\
   \includegraphics[width=6cm,height=4cm]{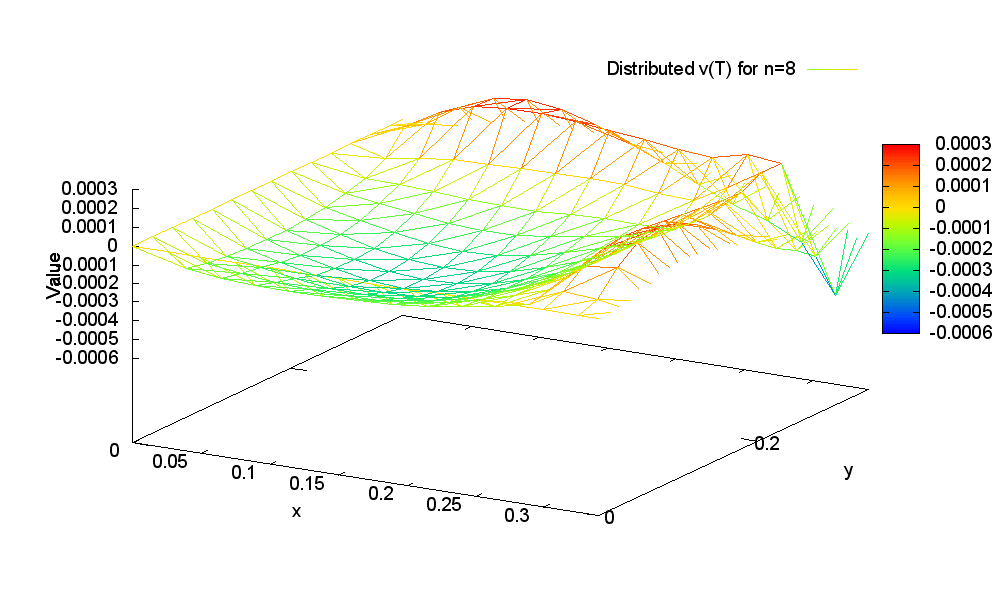} &   \includegraphics[width=6cm,height=4cm]{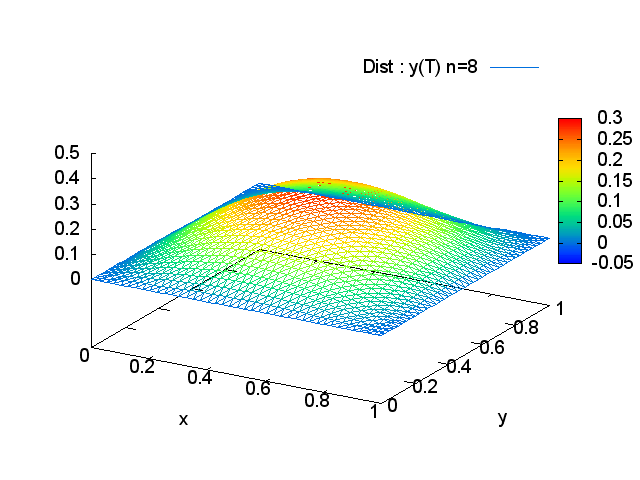} \\        
   \includegraphics[width=6cm,height=4cm]{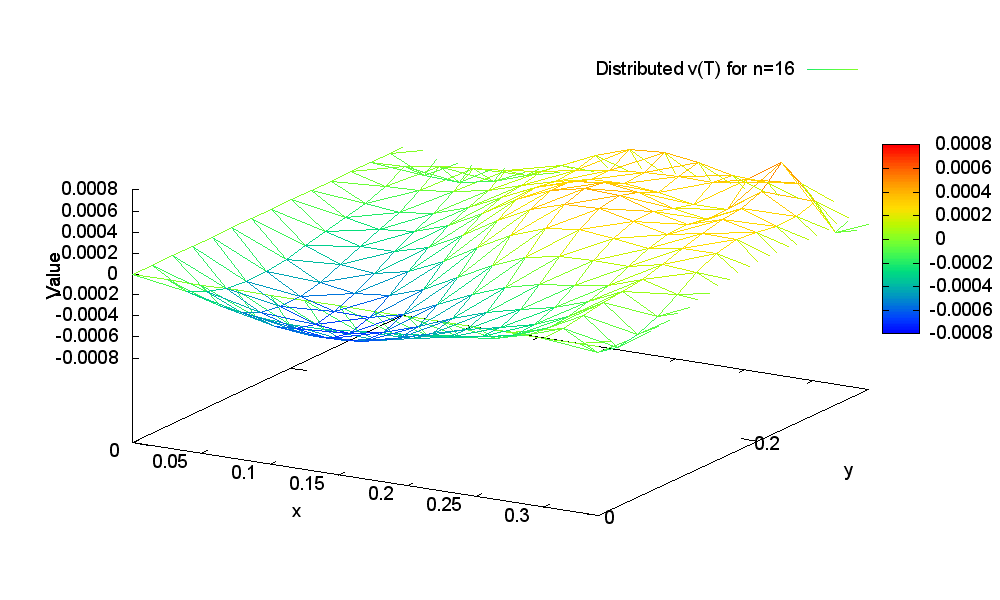} &   \includegraphics[width=6cm,height=4cm]{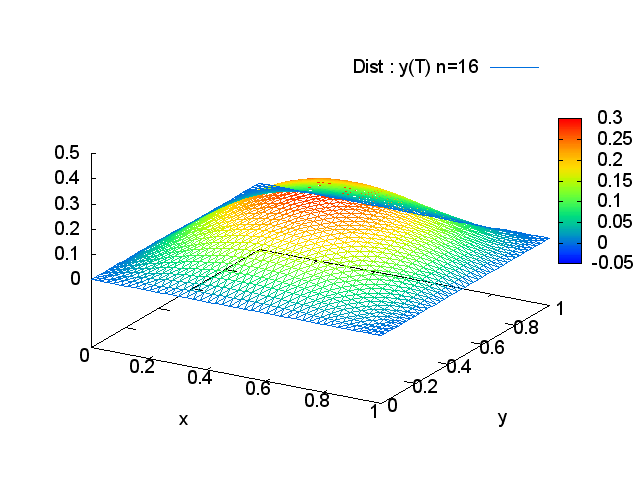} \\           
\end{tabular}
\caption{Several rows value snapshots in $\hat n$ of the distributed optimal control on the left columns and its corresponding controlled final state at time T: $(\mathcal{Y}(T)$ on the right columns. The test case corresponds to the control problem $\mathcal P_2^\alpha$, where $\alpha=1\times 10^{-08}$.}
\end{figure}
\begin{figure}[htbp]
\begin{tabular}{cc}
   \includegraphics[width=6cm,height=4cm]{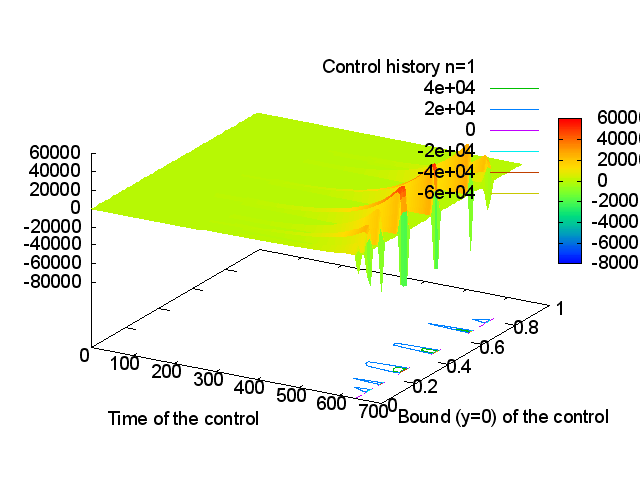} &   \includegraphics[width=6cm,height=4cm]{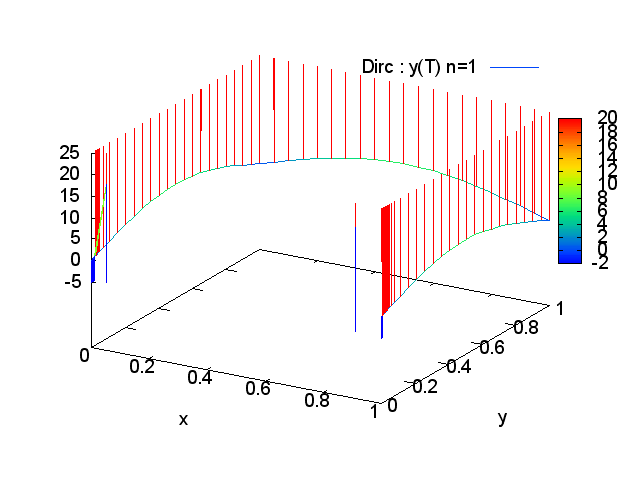} \\
   \includegraphics[width=6cm,height=4cm]{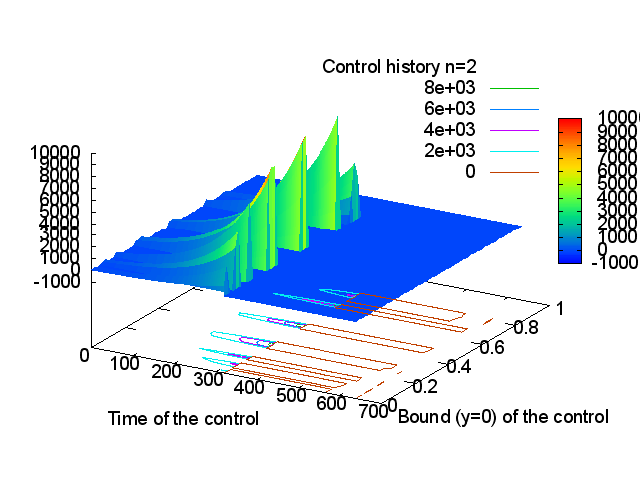} &   \includegraphics[width=6cm,height=4cm]{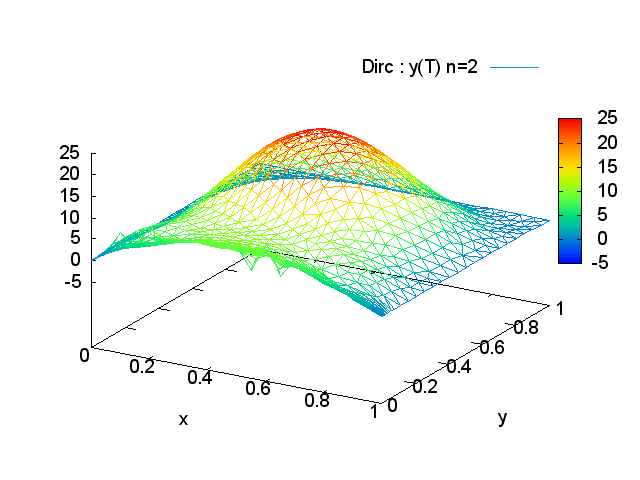} \\
   \includegraphics[width=6cm,height=4cm]{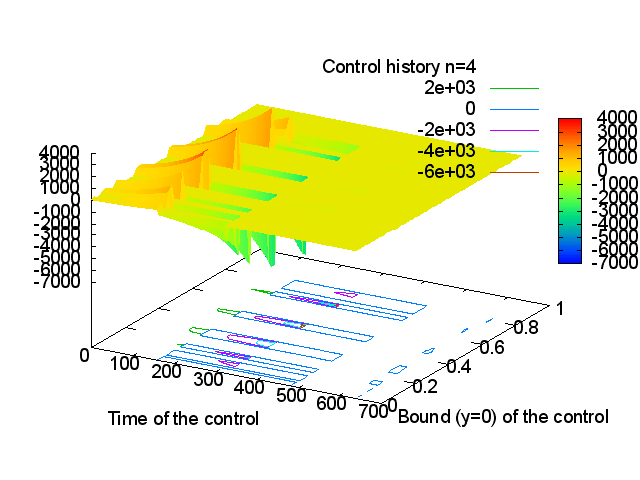} &   \includegraphics[width=6cm,height=4cm]{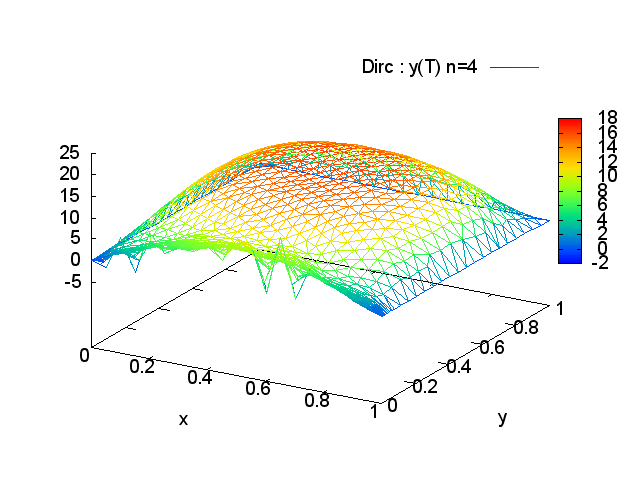} \\
   \includegraphics[width=6cm,height=4cm]{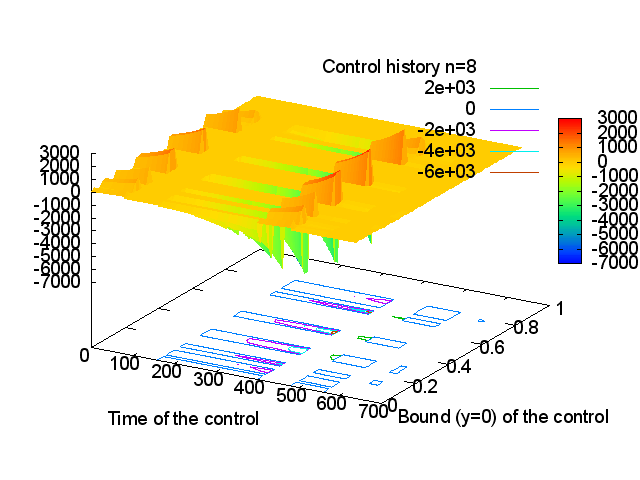} &   \includegraphics[width=6cm,height=4cm]{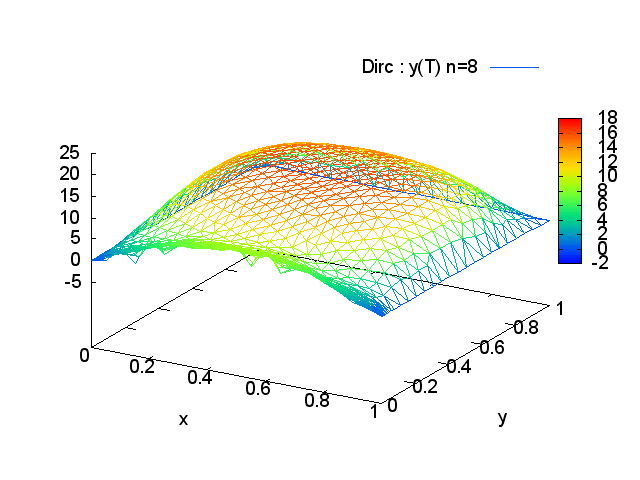} \\        
   \includegraphics[width=6cm,height=4cm]{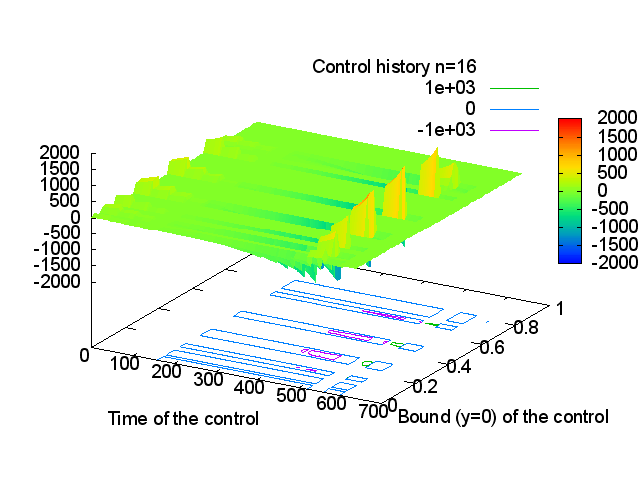} &   \includegraphics[width=6cm,height=4cm]{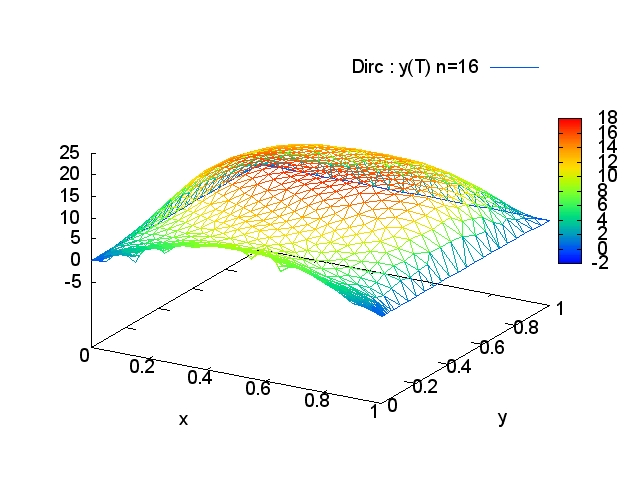} \\           
\end{tabular}
\caption{Several rows value snapshots in $\hat n$ of the Dirichlet optimal control on the left columns and its corresponding controlled final state at time T: $\mathcal{Y}_{\Gamma}(T)$ on the right columns. The test case corresponds to the control problem $\mathcal P_3^\alpha$, where $\alpha=1\times 10^{-08}$..}\label{snapshot_Dirichlet2}
\end{figure}
\begin{table}[htbp]
\tiny
\centering
\begin{tabular}{llccccccc||}
   \hline\hline
Test problem & \multicolumn{3}{c}{Results} \\
   \hline\hline
      $\mathcal P_3^\alpha$ & $\alpha=1\times 10^{-02}$ & & & & & \\
& Quantity & $\hat n=1$ & $\hat n=2$  & $\hat n=4$ & $\hat n=8$ & $\hat n=16$ \\   \hline
&Number of iterations & 40 &40 &30 & 18& 10\\
&walltime in sec  &12453.9&12416.1&9184.28& 5570.54&3158.97 \\
&$\|\mathcal{Y}_{\Gamma}(T)-y^{target}\|_2\slash\|y^{target}\|_2$ &$8.54\times 10^{+06}$&0.472488 & 0.0538509&0.0533826 &0.0534024  \\
&$  \int_{(0,T)}\|v\|_\Gamma^2 dt$ &$2.79\times 10^{+08}$&$1.96\times 10^{+07}$ &31.4193 &138.675 &275.08\\ 
   \hline 
          $\mathcal P_3^\alpha$ & $\alpha=1\times 10^{-08}$ & & & & & \\
& Quantity & $\hat n=1$ & $\hat n=2$  & $\hat n=4$ & $\hat n=8$ & $\hat n=16$ \\   \hline
&Number of iterations &40&40 &30 &27 &10 \\
&walltime in sec  &1248.85 &1248.97 & 916.232&825.791 & 325.16\\
&$\|\mathcal{Y}_{\Gamma}(T)-y^{target}\|_2\slash\|y^{target}\|_2$ &$8.85\times 10^{+06}$ &$0.151086$ & $0.0292072$& $0.0278316$&$0.0267375$  \\
&$  \int_{(0,T)}\|v\|_\Gamma^2 dt$ &$7.92\times 10^{+08}$ &$2.30 \times 10^{+07}$&$1.27\times 10^{+07}$&$1.47\times 10^{+07} $&$1.58\times 10^{+06}$\\\hline
   \end{tabular}
   \caption{Results' summary of Algorithm~\ref{ESDDOC} applied on the Dirichlet boundary control problem $\mathcal P_3^\alpha$.}\label{TabDirichletcontrol}
 \end{table}  

	Table~\ref{TabDistributecontrol} contains the summarized results for the convergence of the distributed control problem. On the one hand, we are interested in the error given by our algorithm for several choices of partition number $\hat n$. On the other hand, we give the $L^2(0,T;L^2(\Omega_c))$ of the control.	We notice the improvement in the quality of the algorithm in terms of both time of execution and control energy consumption, namely the quantity $\int_{(0,T)}\|v^k\|_c^2 dt$. In fact, for the optimal control framework ($\mathcal P_1^\alpha$ and $\mathcal P_2^\alpha$ with $\alpha=1\times 10^{-02}$), we see that, for a fixed stopping criterion, the algorithm is faster and consume the same energy independently of $\hat n$.  In the approximate controllability framework ($\mathcal P_2^\alpha$ with $\alpha=1\times 10^{-08}$ vanishes), we note first that the general accuracy of the controlled solution (see the error ${\|\mathcal{Y}^k(T)-y^{target}\|_2\slash\|y^{target}\|_2}$) is improved as $\alpha=1\times 10^{-08}$ compered with $\alpha=1\times 10^{-02}$. Second, we note that the error diminishes when increasing $\hat n$, the energy consumption rises however. The scalability in CPU's time and number of iteration shows the enhancement of our method when it is applied (i.e. for $\hat n>1$). 	

	Table~\ref{TabDirichletcontrol} contains the summarized results at the convergence of the Dircichlet boundary control problem. This problem is known in the literature for its ill-posedness, where it may be singular in several cases see~\cite{MR2793831} and references therein. In fact, it is very sensitive to noise in the data. We show in Table~\ref{TabDirichletcontrol} that for a big value of the regularization parameter $\alpha$ our algorithm behaves already as the distributed optimal control for a vanishing $\alpha$, in the sense that it consumes more control energy to produce a more accurate solution with smaller execution CPU's time. It is worth noting that the serial case $\hat n=1$ fails to reach an acceptable solution, whereas the algorithm behaves well as $\hat n$ rises.

	We give in Figure.~\ref{snapshot_Dirichlet1} and Figure.~\ref{snapshot_Dirichlet2} several rows value snapshots (varying  $\hat n$) of the Dirichlet control on $\Gamma$. We present in the first column its evolution during $[0,T]$ and on the second column its corresponding controlled final solution $y(T)$ at time $T$; we scaled the plot of the z-range of the target solution in both Figs.\ref{snapshot_Dirichlet1} and \ref{snapshot_Dirichlet2}. 
	  
	 In each row one sees the control and its solution for a specific partition $\hat n$. The serial case $\hat n=1$ leads to a controlled solution which doesn't have the same rank as $y^{target}$, whereas as $\hat n$ rises, we improve the behavior of the algorithm.

	 It is worth noting that the control is generally active only around the final horizon time $T$. This is very clear in Figure.~\ref{snapshot_Dirichlet1} and Figure.~\ref{snapshot_Dirichlet2} (see the first row i.e. case $\hat n=1$). The nature of our algorithm, which is based on time domain decomposition, obliges the control to act in subintervals.  Hence, the control acts more often and earlier in time (before $T$) and leads to a better controlled solution $y(T)$.
\subsubsection{Third test problem: Sever ill-posed problem (no solution)}
	In this test case, we consider a severely ill-posed problem. In fact, the target solution is piecewise Lipschitz continuous, so that it is not regular enough compared with the solution of the heat equation. This implies that in our control problem, both the distributed and the Dirchlet boundary control has no solution. The initial condition and the target solution are given by 
	\begin{equation}\tag{$\mathcal P_4^{\alpha}$}
	\begin{split}
	 y_0(x_1,x_2) &= \pi(sin(\pi x_1) + sin(\pi x_2))\\
	y^{target}(x_1,x_2)  &= \min{\big(x_1,x_2,(1-x_1),(1-x_2)\big)},
	\end{split}
	\end{equation}
	respectively. A plots of the initial condition and the target solutions are given in Figure.~\ref{P4figTARINIT}.
\begin{figure}[htbp]
   \begin{minipage}[c]{.46\linewidth}
      \includegraphics[width=6cm,height=6cm]{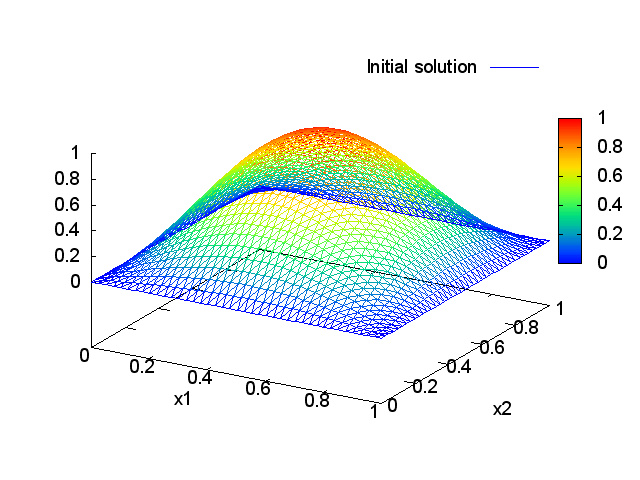}
   \end{minipage} \hfill
   \begin{minipage}[c]{.46\linewidth}
      \includegraphics[width=6cm,height=5.6cm]{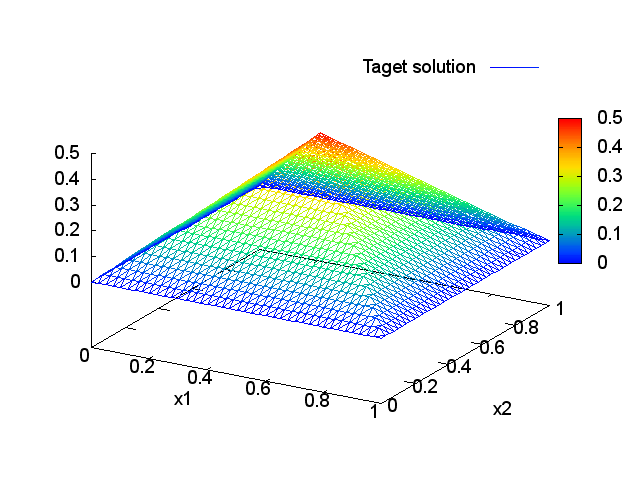}
   \end{minipage}
   \caption{ Graph of initial and target solution for both distributed and Dirichlet boundary control problem.}\label{P4figTARINIT}
\end{figure}
\begin{figure}[htpb]
\begin{tabular}{cc}
   \includegraphics[width=6cm,height=4cm]{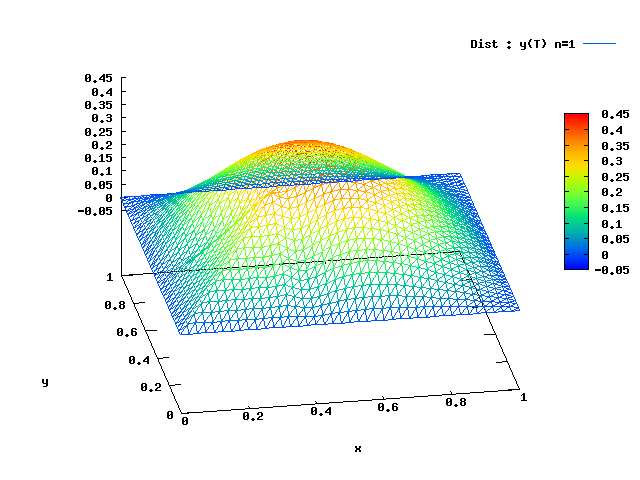} &
   \includegraphics[width=6cm,height=4cm]{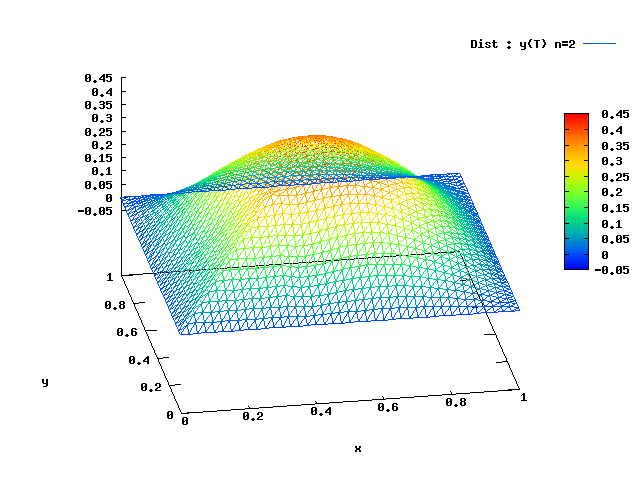} \\
   \includegraphics[width=6cm,height=4cm]{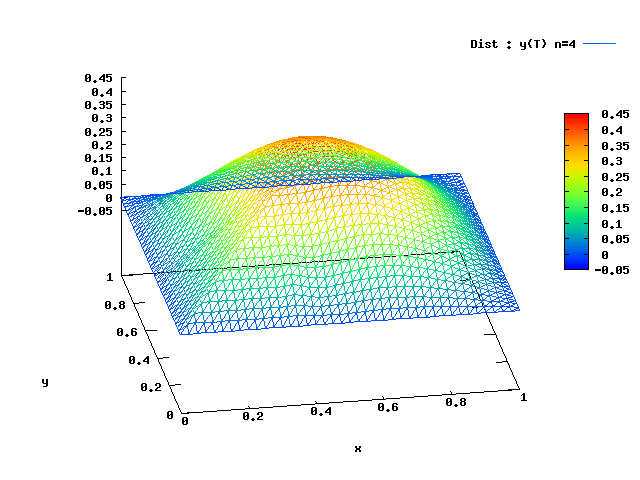} &
   \includegraphics[width=6cm,height=4cm]{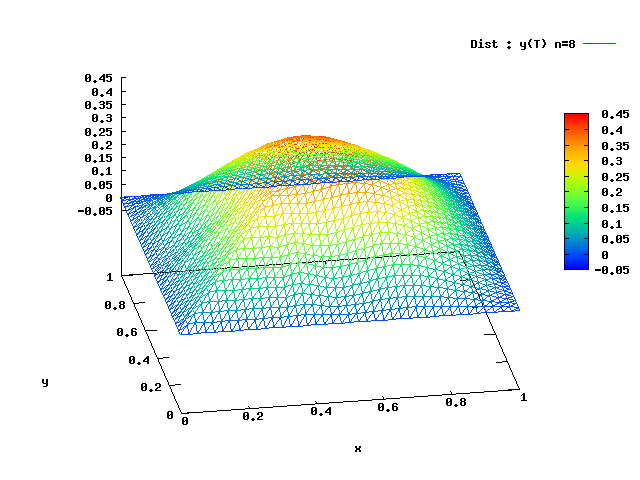} \\        
   \includegraphics[width=6cm,height=4cm]{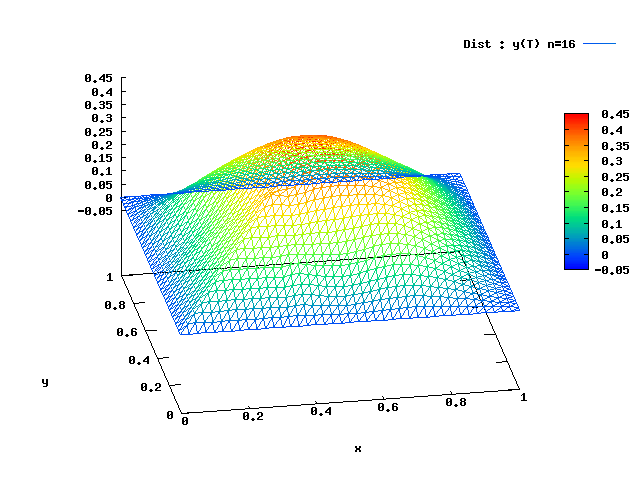} & 
     \includegraphics[width=6cm,height=4cm]{sever_uT.png}        
\end{tabular}
\caption{Several snapshots in $\hat n$ of final state at time T: $\mathcal{Y}(T)$. The test case corresponds to Distributed control sever Ill-posed problem $\mathcal P^{\alpha}_4$.}\label{P4DIS}
\end{figure} 
\begin{figure}[htbp]
\begin{tabular}{cc}
   \includegraphics[width=6cm,height=4cm]{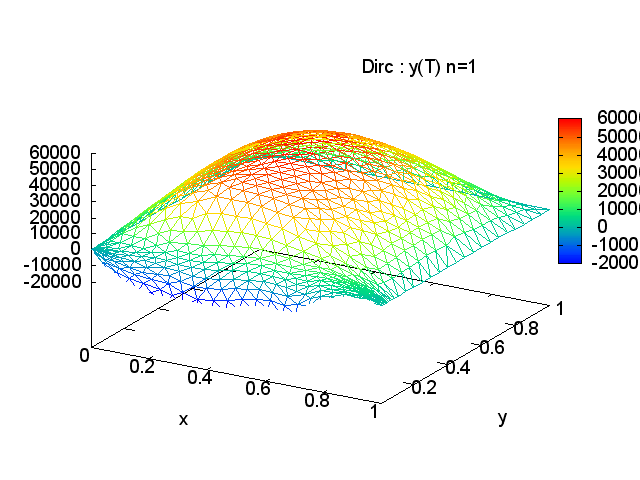} &
   \includegraphics[width=6cm,height=4cm]{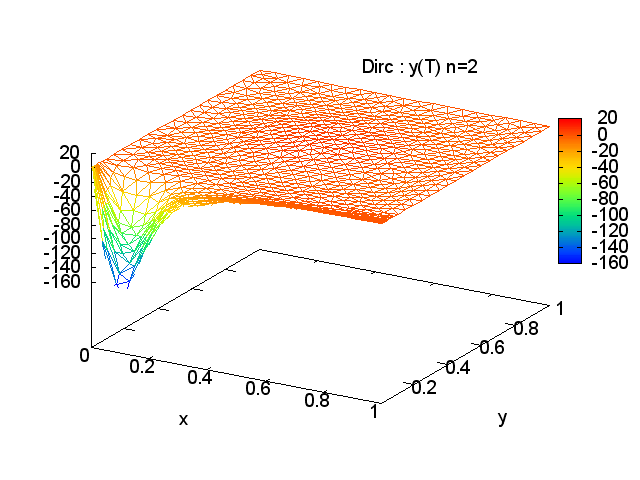} \\
   \includegraphics[width=6cm,height=4cm]{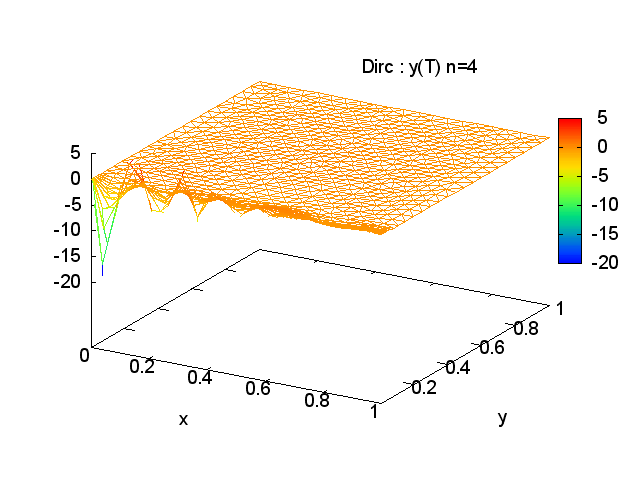} &
   \includegraphics[width=6cm,height=4cm]{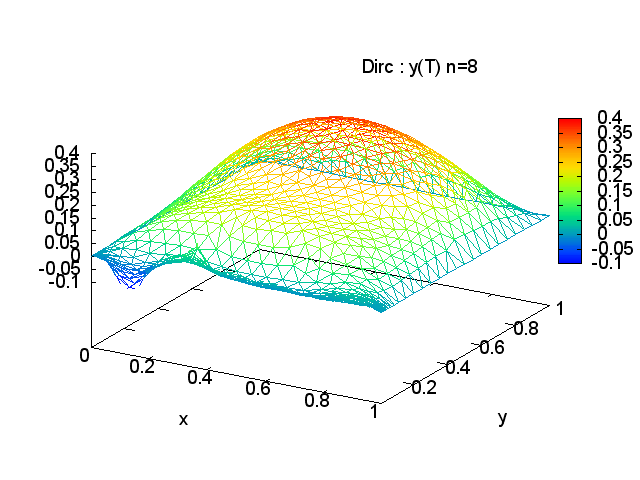} \\        
   \includegraphics[width=6cm,height=4cm]{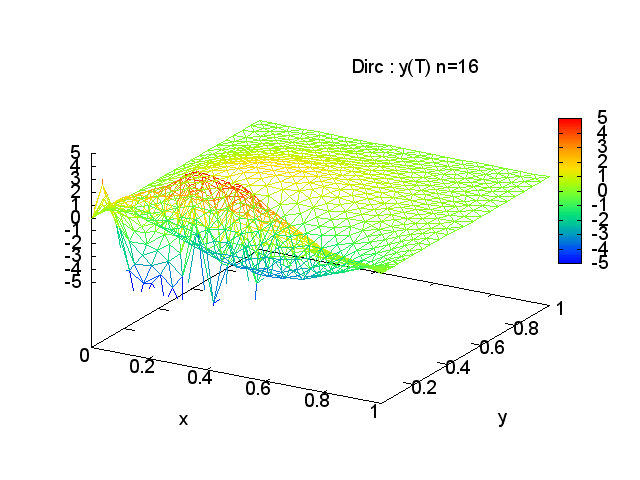}& 
     \includegraphics[width=6cm,height=4cm]{sever_uT.png}        
\end{tabular}
\caption{Several snapshots in $\hat n$ of final state at time T: $\mathcal{Y}_{\Gamma}(T)$. The test case corresponds to Dirichlet control sever Ill-posed problem $\mathcal P^{\alpha}_4$.}\label{P4DIR}
\end{figure} 
\begin{table}[htbp]
\begin{centering}
\tiny
\begin{tabular}{llccccccc||}
   \hline\hline
Test problem & \multicolumn{3}{c}{Results} \\
   \hline\hline
    Distributed control & & & & & & \\
   $\mathcal P_4^\alpha$&$\alpha=1\times 10^{-08}$ & & & & & \\
& Quantity & $\hat n=1$ & $\hat n=2$  & $\hat n=4$ & $\hat n=8$ & $\hat n=16$ \\   \hline
&Number of iterations & 100 &68&60 & 50&40 \\
&walltime in sec&6381.43 & 6303.67  & 5548.16 &4676.83  & 3785.97\\
&$\|\mathcal{Y}(T)-y^{target}\|_2\slash\|y^{target}\|_2$ & $8.16\times 10^{-03}$ &$5.3\times10^{-03}$& $4.74\times10^{-03}$& $3.95\times10^{-03}$&$3.76\times10^{-03}$   \\
&$  \int_{(0,T)}\|v\|_c^2 dt$ &0.34 & 3.01 &52.87 & 52.77 & 2660.87\\ 
   \hline 
    Dirichlet control & & & & & & \\
         $\mathcal P_4^\alpha$ &$\alpha=1\times 10^{-08}$ & & & & & \\
& Quantity & $\hat n=1$ & $\hat n=2$  & $\hat n=4$ & $\hat n=8$ & $\hat n=16$ \\   \hline
&Number of iterations & 25 &25 & 20 & 4 & 1 \\
&walltime in sec &848.58 & 655.40 & 655.40 & 146.19 & 62.87 \\
&$\|\mathcal{Y}_{\Gamma}(T)-y^{target}\|_2\slash\|y^{target}\|_2$ & $2.85\times 10^{+10}$ &3055 & 39.3 & 0.2 & 0.067  \\
&$  \int_{(0,T)}\|v\|_\Gamma^2 dt$ &$6.73\times 10^{+08}$&$2.17\times 10^{+07}$ & 141.62 & 17.84 &26758.5\\\hline
   \end{tabular}
   \caption{Results' summary of Algorithm~\ref{ESDDOC} applied on to both distributed and Dirichlet boundary control for the third test problem $\mathcal P_4^\alpha$.}\label{Tab3}
   \end{centering}
 \end{table}  

	In Figures~\ref{P4DIS} and~\ref{P4DIR} we plot the controlled solution at time $T$ for the distributed and Dirichlet control problems respectively. We remark that for the distributed control problem the controlled solution is smooth except in $\Omega_c$, where the control is able to fit with the target solution. 
\begin{remark}	
	Out of curiosity, we tested the case where the control is distributed on the whole domain. We see that the control succeeds to fit the controlled solution to the target even if it is not $\mathcal{C}^1(\Omega)$. This is impressive and shows the impact on the results of the regions where the control is distributed. 
\end{remark}
	 We note the stability of the method of the distributed test case. However, the Dirichlet problem test case presents hypersensitivity. In fact, in the case of $\hat n=1$ the algorithm succeeds to fit an acceptable shape of the controlled solution, although still far in the scale. We note that the time domain decomposition leads to a control which gives a good scale of the controlled solution.

	 In this severely ill-posed problem, we see that some partitions may fail to produce a control that fit the controlled solution to the target. There is an exemption for the case of $\hat n=8$ partitions, where we have a good reconstruction of the target. The summarized results are given in Tables~\ref{Tab3}.

\subsubsection{Regularization based on the optimal choice of partition}

	 The next discussion concerns the kind of situation where the partition leads to multiple solutions, which is common in ill-posed problems. In fact, we discuss a regularization procedure used as an exception handling tool to choose the best partition, giving the best solution of the handled control problem.
	
	It is well known that ill-posed problems are very sensitive to noise, which could be present due to numerical approximation or to physical phenomena. In that case, numerical algorithm may blow-up and fail. We present several numerical tests for the Dirichlet boundary control, which is a non trivial problem numerically. The results show that in general time domain decomposition may improve the results in several cases. But scalability is not guaranteed as it is for the distributed control. We propose a regularization procedure in order to avoid the blow-up and also to guarantee the optimal choice of partition of the time domain. This procedure is based on a test of the monotony of the cost function. In fact, suppose that we possess 64 processors to run the numerical problem. Once we have assembled the Hessian $H_k$ and the Jacobian $D_k$ for the partition $\hat n=64$, we are actually able to get for free the results of the Hessian and the Jacobian for all partitions $\hat n$ that divide 64. Hence, we can use the quadratic property of the cost functional in order to predict and test the value of the cost function for the next iteration without making any additional computations. The formulae is given by:
\begin{equation*} J(v^{k+1})= J(v^{k})-\frac 1 2 D_{k}^TH_{k}^{-1}D_{k}.\end{equation*}
We present in Algorithm~\ref{reduceHD} the technique that enables us to reduce in rank and compute a series of Hessians and Jacobians for any partition $\hat  n$ that divide the available number of processors. An exemple of the applicability of these technique, on a $4$-by-$4$ SPD matrix, is given in Appendix.
\begin{algorithm}[htbp]
\SetAlgoLined\LinesNumbered
\ShowLn
\KwIn{$\hat n, {H}_{\hat n}^k, D_{\hat n}^{k}$\;}
$n=\hat n$\;
$J_{n\slash2}^{k+1}=J_{ n}^{k+1}$\;
\While{$J_{n\slash2}^{k+1}>J_{n}^{k}$}
{
\For{$i=0;i\leq n; i+2$}
{
	$\big({  D}_{n\slash2}^k\big)_{i} = \big(D_n^k\big)_{i}+\big(D_n^k\big)_{i+1}$\;
	\For{$j=0;j\leq n; j+2$}
	{
		$\big({  H}_{n\slash2}^k\big)_{i,j} = \big(H_n^k\big)_{j}+\big(H_n^k\big)_{j+1}  $\;
	}
}
Estimation of the cost $J_{n\slash2}^{k}$\; 
$n=n \slash 2$\;
}
\caption{Reduce in rank of the partition $\hat n$}\label{reduceHD}
\end{algorithm}

\section{Conclusion}
\label{riahi_mini_15_sec:5}
We have presented in this article a new optimization technique to enhance the steepest descent algorithm via domain decomposition in general and we applied our new method in particular to time-parallelizing the simulation of an optimal heat control problem. We presented its performance (in CPU time and number of iterations) versus the traditional steepest descent algorithm in several and various test problems. The key idea of our method is based on a quasi-Newton technique to perform efficient real vector step-length for a set of descent directions regarding the domain decomposition. The originality of our approach consists in enabling parallel computation where its vector step-length achieves the optimal descent direction in a high dimensional space. Convergence property of the presented method is provided. Those results are illustrated with several numerical tests using parallel resources with MPI implementation.  
	

 \appendix

\section{Kantorovich matrix inequality}
	For the sake of completeness, we give in this appendix the Matrix Kantorovich inequality, that justifies the statement of our convergence proof.  Assume that $\nabla^2 J$ is symmetric positive definite with smallest and largest eigenvalues $\lambda_\text{min}$ and $\lambda_\text{max}$ respectively. We give in the following the matrix version of the famous Kantorovich inequality, which reads:
\begin{theorem}[see~\cite{MR0053389} for more details]
Assume that $\sum_{n=1}^{\hat n}\alpha_{n}=1$ where $\alpha_{n}\geq 0$ and $\lambda_n>0\quad \forall n$; we have thus : 
$$\sum_{n=1}^{\hat n} \alpha_{n}\lambda_{n}\sum_{n=1}^{\hat n}\frac{\alpha_{n}}{\lambda_{n}} 
\leq
\frac{(\lambda_\text{max}+\lambda_\text{min})^2}{4\lambda_\text{max}\lambda_\text{min}}.$$
\end{theorem}
By diagonalizing the symmetric positive definite operator $H$ we obtain: $H=P\Lambda P^{-1}$, where $P$ is orthonormal operator (i.e. $P^{T}=P^{-1}$). Recall Eq.\eqref{kant-matrix} that we rewrite as:
$$
	\frac{ \|\nabla  J(v^k)\|^2_{\nabla^2 J} \|\nabla  J(v^k)\|^2_{(\nabla^2 J)^{-1}}}{ \|\nabla  J(v^k)\|^4_c }
	\leq
 \frac{({\lambda_{max}}+{\lambda_{min}})^2}{4\lambda_{max}\lambda_{min}}.
$$
	In order to simplify the expression, we shall use $d_k$ instead of $\nabla  J(v^k)$  so that the equation above reads:
\begin{eqnarray*}
\frac{d_{k}^{T}(\nabla^2 J)d_{k}\,\,d_{k}^{T}(\nabla^2 J)^{-1}d_{k}}{(d_{k}^{T}d_{k})^2}=\frac{d_{k}^{T}P^{T} \Lambda Pd_{k}}{d_{k}^{T}P^{T} Pd_{k}}
       \frac{d_{k}P^{T} \Lambda^{-1} Pd_{k}}{d_{k}^{T}P^{T} Pd_{k}}.
\end{eqnarray*}
Let us define ${\bf d}_{k}:=Pd_{k}$, consequently the above equality becomes:
\begin{eqnarray*}
\frac{{\bf d}_{k}^T \Lambda {\bf d}_{k}}{{\bf d}_{k}^{T} {\bf d}_{k}}
       \frac{{\bf d}_{k}^{T} \Lambda^{-1} {\bf d}_{k}}{{\bf d}_{k}^{T}{\bf d}_{k}} 
       =\sum_{n=1}^{\hat n}\frac{({\bf d}_{k})_{n}^{2} }{{\bf d}_{k}^{T}{\bf d}_{k}}\lambda_{n}
\sum_{n=1}^{\hat n} \frac{({\bf d}_{k})_{n}^{2} }{{\bf d}_{k}^{T}{\bf d}_{k}}\frac{1}{\lambda_{n}}.
\end{eqnarray*}
We then denote by $\alpha_n=\frac{({\bf d}_{k})_{n}^{2} }{{\bf d}_{k}^{T}{\bf d}_{k}}$ so that $\sum_{n=1}^{\hat n}\alpha_n=1$, and finally:
 $$
\frac{d_{k}^{T}Ad_{k}\,\,d_{k}^{T}A^{-1}d_{k}}{(d_{k}^{T}d_{k})^2} = \sum_{n=1}^{\hat n} \alpha_{n}\lambda_{n}\sum_{n=1}^{\hat n}\frac{\alpha_{n}}{\lambda_{n}}.
 $$

$\hfill\square$

\begin{example}
Exemple 4-by-4 SPD matrix reduced in rank using the regularization procedure described in Algorithm~\ref{reduceHD}. In order to illustrate the steps of Algorithm 4, we choose a simple example: a matrix 4-by-4 which we are going to reduce recursively in 2-by-2 and in 1-by-1 as follows:
\begin{equation*}
\left(\begin{array}{cccc}
6 & 1 & 2 & 3 \\
1 & 8 & 2 & 4 \\
2 & 2 & 12 & 7 \\
3 & 4 & 7 & 16
\end{array}\right)
\mapsto
\left(\begin{array}{cccc}
(6 & 1) & (2 & 3) \\
(1 & 8) & (2 & 4) \\
(2 & 2) & (12 & 7) \\
(3 & 4) & (7 & 16)
\end{array}\right)
\mapsto
\left(\begin{array}{cc}
7 & 5 \\
9 & 6\\
4 & 19 \\
7 & 23\\ 
 \end{array}\right)
 \mapsto
 \left(\begin{array}{cc} 16 & 11 \\11 & 42 \end{array}\right)
\end{equation*}

\begin{equation*}
 \left(\begin{array}{cc} 16 & 11 \\11 & 42 \end{array}\right)\mapsto
  \left(\begin{array}{c} 27 \\ 53 \end{array}\right) \mapsto  \left(80\right)
\end{equation*}

\end{example}
\bibliographystyle{plain}

\bibliography{biblio}


%

\end{document}